\documentclass[12pt]{amsart}
\usepackage{amssymb,enumerate}

\makeatletter % for internal macros with @
\renewcommand\p@enumii{}
\makeatother

\newcommand \R{\mathbb R}

\newcommand \ad{\operatorname{ad}}

\newcommand \tr{\operatorname{tr}}
\newcommand \ric{\operatorname{ric}}
\newcommand \Ric{\operatorname{Ric}}
\newcommand \End{\operatorname{End}}
\newcommand \rk{\operatorname{rk}}
\newcommand \Span{\operatorname{Span}}
\renewcommand \a{\mathfrak a}
\newcommand \ig{\mathfrak i}
\newcommand \g{\mathfrak g}
\newcommand \h{\mathfrak h}
\newcommand \m{\mathfrak{m}}
\newcommand \z{\mathfrak z}

\newcommand \<{\langle}
\renewcommand \>{\rangle}
\newcommand \ip{\<\cdot,\cdot\>}
\newcommand \G{\mathfrak G}
\newcommand \bu{\mathbf{U}}

\theoremstyle{plain}
\newtheorem{theorem}{Theorem}

\newtheorem{lemma}{Lemma}

\theoremstyle{definition}

\theoremstyle{remark}
\newtheorem{remark}{Remark}

\begin{document}
\title[Curvature of metric nilpotent Lie algebras]{Curvature properties of metric nilpotent Lie algebras which are independent of metric}
%Subsets of a particular sign of the curvature and of extremal Ricci curvature of metric nilpotent Lie algebras
%Subsets of stable sign and of extrema of the curvature in metric nilpotent Lie algebras
%On the curvature of nilpotent Lie groups
\author[G. Cairns, A. Hini\'c Gali\'c, Y. Nikolayevsky]{Grant Cairns, Ana Hini\'c Gali\'c and Yuri Nikolayevsky}


\address{Department of Mathematics and Statistics, La Trobe University, Melbourne,
Australia 3086}
\email{G.Cairns@latrobe.edu.au}
\address{Department of Mathematics and Statistics, La Trobe University, Melbourne,
Australia 3086}
\email{A.HinicGalic@latrobe.edu.au}
\address{Department of Mathematics and Statistics, La Trobe University, Melbourne,
Australia 3086}
\email{Y.Nikolayevsky@latrobe.edu.au}

\thanks{The authors were partially supported by ARC Discovery grant DP130103485}

%\date{\today}

\begin{abstract}
This paper consists of two parts. First, motivated by classic results, we determine the subsets of a given nilpotent Lie algebra $\g$ (respectively, of the Grassmannian of two-planes of $\g$) whose sign of Ricci (respectively, sectional) curvature remains unchanged for an arbitrary choice of a positive definite inner product on $\g$. In the second part we study the subsets of $\g$ which are, for some inner product, the eigenvectors of the Ricci operator with the maximal and with the minimal eigenvalue, respectively. We show that the closures of these subsets is the whole algebra $\g$, apart from two exceptional cases: when $\g$ is two-step nilpotent and when $\g$ contains a codimension one abelian ideal.
\end{abstract}

\maketitle

\section{Introduction}
\label{s: intro}

In the classic paper of Milnor \cite{Mi} it was shown that all non-abelian nilpotent Lie groups $G$ have some positive curvature and some negative curvature. The context here is that of left-invariant Riemannian metrics, so these are determined by a choice of inner product on the Lie algebra $\g$ of $G$. More explicitly,
Milnor showed that \cite[Lemmas 2.1, 2.3]{Mi}:
\begin{enumerate}
\item  for all $X$ in the centre $\z$ of $\g$, the sectional curvature satisfies $K(X,Y)\geq 0$ for all $Y \in \g$;
\item  for all $X$ orthogonal to the derived algebra $\g'=[\g,\g]$, the Ricci curvature satisfies $\Ric(X)\leq 0$.
\end{enumerate}
Examining some common nilpotent Lie algebras by taking the basis elements used for their presentations to be orthonormal, the impression one obtains is that the positive curvature is typically concentrated ``near the centre'', while the negative curvature is found at the ``upper levels of the algebra''. The aim of this paper is to explore the veracity of this maxim. We present some rather surprising results in Theorems~\ref{T:everymetric} and \ref{T:maxmin} below.

We begin by exploring both the sectional curvature and the Ricci curvature, and we consider two variants of the problem: properties that hold for some inner product, and  properties that hold for all inner products. For the first variant, given a nilpotent Lie algebra $\g$, we denote by $\g_>,\,\g_{\ge},\,\g_0,\,\g_{\le}$, and $\g_<$ the subsets of vectors $X \in \g$ such that for every choice of the inner product on $\g$, the Ricci curvature $\Ric(X)$ is positive, nonnegative, zero, nonpositive, and negative, respectively. Similarly, we denote by
$\G_>,\,\G_{\ge},\,\G_0,\,\G_{\le}$, and $\G_<$ the subsets of all two-planes $\sigma=\Span(X,Y)$ in the Grassmannian $G(2,\g)$ such that for any choice of the inner product on $\g$, the sectional curvature $\kappa(X,Y)$ is positive, nonnegative, zero, nonpositive, and negative, respectively. Clearly, $\g_> \subset \g_{\ge},\,\g_< \subset \g_{\le}, \, \g_0 \subset (\g_{\ge} \cap \g_{\le})$ (and the same is true with $\g$ replaced by $\G$) and, for an abelian algebra, $\g_>=\g_<=\varnothing$ and $\g_{\ge}=\g_0=\g_{\le}=\g$. It was proved in \cite{Gotoh} that $\g_{\ge}=\z$. Let $G(2,\g)$ (resp.~$G(2,\z)$) denote the Grassmannian of the two-planes lying in $\g$ (resp.~$\z$). We have the following theorem.

\begin{theorem}\label{T:everymetric}
Let $\g$ be a nonabelian nilpotent Lie algebra. Then
\begin{enumerate}[{\rm (a)}]
\item \label{it:everyR}
$ \g_< = \varnothing, \quad \g_{\le}=\g_0=\{0\},\quad \g_{\ge}=\z, \quad \g_>=(\g' \cap \z) \setminus \{0\}$.

\item  \label{it:everys}
$\begin{aligned}[t]
    \G_< &= \varnothing, \quad \G_{\le}=\G_0=G(2,\z),\\
    \G_{\ge} &=\{\Span(X,Y)\in G(2,\g) \, : \, [X,Y]=0 \; \& \; \forall Z \in \g, [X,Z] \parallel [Y,Z] \} \\
    &=\{\sigma\in G(2,\g) \, : \, [\sigma,\sigma]=0 \; \& \; \forall Z \in \g, \exists X \in \sigma: [X,Z]=0 \}, \\
    \G_> &=\{\Span(X,Y) \in G(2,\g) \, : \, X \in [Y,\g] \cap \z\}.
\end{aligned}$
\end{enumerate}
\end{theorem}

\begin{remark} \label{rem1}
It follows from Theorem~\ref{T:everymetric}\eqref{it:everys} that if $\g$ is a nilpotent Lie algebra and $X\in\g$, then $X\in\z$ if and only if $\Span(X,Y)\in \G_{\ge}$ for all $Y\in\g$. This fact is true for all Lie algebras; it was conjectured by Milnor \cite{Mi} and proved in \cite{Abib, Ue}.
A more explicit (and somewhat nicer) description of the set $\G_{\ge}$ will be given in Lemma~\ref{L:Gge} in Section~\ref{S:signstable}.
\end{remark}

Our second result shows that, with a few exceptions, the Ricci curvature of a nilpotent Lie algebra can attain its maximum and its minimum on almost every vector, for appropriate choices of inner product. We consider the Ricci operator $\ric \in \End(\g)$ defined by $\< \ric X, Y \>=\Ric(X,Y)$, for $X, Y \in \g$, and we examine the maximal and the minimal eigenvalues of $\ric$ and the corresponding eigenvectors. For a linear space $L$, denote $\mathbb{P}L$ the projective space over $L$, and $\pi: L \setminus \{0\} \to \mathbb{P}L$ the natural projection. A point $u \in \mathbb{P}\g$ is called \emph{Ricci-maximal} (respectively, \emph{Ricci-minimal}), if there exists an inner product $\ip$ on $\g$ such that a vector $X \in \pi^{-1}(u) \subset \g$ is an eigenvector of the Ricci operator $\ric$ for $\ip$ with the maximal (respectively, minimal) eigenvalue. We have the following theorem.

\begin{theorem}\label{T:maxmin}
Let $\g$ be a nilpotent Lie algebra. Let $M, m \subset \mathbb{P}\g$ be the sets of Ricci-maximal and Ricci-minimal
points, respectively.

\begin{enumerate}[{\rm (a)}]
\item \label{IT:max}
\begin{enumerate}[{\rm (i)}]
  \item \label{IT:2step}
  If $\g$ is two-step nilpotent, then $\overline{M}=\mathbb{P}\g'$.
  \item \label{IT:codim1}
  If $\g$ has a codimension one abelian ideal $\a$ and is not two-step nilpotent, then $\overline{M}=\mathbb{P}\a$.
  \item \label{IT:other}
  In all the other cases, $\overline{M}=\mathbb{P}\g$.
\end{enumerate}

\item \label{IT:min}
$\overline{m}=\mathbb{P}\g$.
\end{enumerate}

\end{theorem}

Let $\g$ be a Lie algebra. For a subalgebra $\h \subset \g$, its derived algebra and the centre are denoted by $\h'$ and $\z(\h)$ respectively (and we replace $\z(\g)$ by simply $\z$). We use $\oplus$ for the direct sum of linear spaces, not of Lie algebras (even when both summands are Lie algebras). When we say that a Lie algebra is defined by certain relations between basis elements, all the brackets which are not listed (and do not follow from the listed ones by skew-symmetry) are assumed to be zero. We say that a certain condition is satisfied for \emph{almost all} elements of a topological space, if it is satisfied for a dense subset of elements (in the most cases through the paper it will also be open).

The paper is organised as follows: after giving brief preliminaries in Section~\ref{s:pre}, we prove Theorem~\ref{T:everymetric} in Section~\ref{S:signstable}. The proof of Theorem~\ref{T:maxmin} given in Section~\ref{S:extrema} relies on a series of lemmas whose proves are given in Sections~\ref{S:lemma5} and \ref{S:lemmata}.

We remark that there have been several recent papers that have investigated the curvature properties of Lie groups; see \cite{Cheb,Ch,KN1,KN2,Ni}.

The authors gratefully acknowledge the contribution of Marcel Nicolau (Barcelona). 

%We employ the classification of real nilpotent Lie algebras of dimension up to $6$, which is well established and has been verified by several authors. We use the classification of Nielsen \cite{Niel}; see also \cite{dG}.
%For further recent papers on the geometry of metric Lie algebras, see \cite{CHN1,CHN2,CLNN,CKM,KS,Mc}.

\section{Preliminaries}
\label{s:pre}

Let $G$ be a Lie group with a left-invariant metric. The latter is completely determined by an inner product $\ip$ on the Lie algebra $\g$ of $G$. It is well known that the sectional curvature of the two-plane $\sigma=\Span(X,Y), \; X, Y \in \g$, is given by $K(X,Y) \|X \wedge Y\|^{-2}$, where
\begin{equation}\label{E:sect}
\begin{split}
K(X,Y)&=\|\bu(X, Y)\|^2 - \<\bu(X, X), \bu(Y, Y)\> -\tfrac34 \|[X, Y]\|^2 \\
    &\hphantom{=\|\bu(X, Y)\|^2 \;}-\tfrac12 \<[X, [X, Y]], Y\> -\tfrac12 \<[Y, [Y, X]], X\>,
\end{split}
\end{equation}
and $\<\bu(V,W),Z\> = \tfrac12(\<V, [Z, W]\> + \<W, [Z, V]\>$.

From this one can easily obtain the formula for the Ricci curvature (which is also well known). In particular, if $\g$ is nilpotent and $\{e_1, \dots, e_n\}$ is an orthonormal basis for $(\g, \ip)$, then the Ricci curvature is given by
\begin{align}\label{E:RicXY}
\Ric(X,Y)&=\frac14\sum_{i,j}\langle[e_i,{e_j}],X\rangle \langle[e_i,{e_j}],Y\rangle -\frac12\sum_{i}\langle[X,e_i],[Y,{e_i}]\rangle, \\
\Ric(X)&=\frac14\sum_{i,j}\langle X,[e_i,e_j]\rangle^2 -\frac12\sum_{i}\Vert[X,e_i]\Vert^2. \label{E:RicX}
\end{align}

We will need the following (generally known) lemma, the proof of which we postpone until Section \ref{S:lemmata}.

\begin{lemma}\label{L:general}
Let $\g$ be a nilpotent Lie algebra. Then
\begin{enumerate}[{\rm (a)}]
  \item \label{IT:br} if $[X, Y] \in \Span(X, Y)$ for some $X, Y \in \g$, then $[X,Y]=0$.
  \item \label{IT:2br} if $[X,[X, Y]] \in \Span(X, Y, [X,Y])$ for some $X, Y \in \g$, then $[X,[X, Y]]=0$.
 \item \label{IT:two} if $[X,[X, Y]]=0$, for all $X, Y \in \g$, then $\g$ is two-step nilpotent.
\end{enumerate}
\end{lemma}

\section{Proof of Theorem~\ref{T:everymetric}}
\label{S:signstable}

Starting from the classic results of \cite[Corollary~1.3, Lemma~2.3, Theorem~2.5]{Mi}, in this section we study the following question: for which vectors (respectively two-planes) in a nilpotent Lie algebra, does the Ricci curvature (respectively sectional curvature) have the same sign, regardless of the choice of inner product?
% need this?

\begin{proof}[Proof of Theorem~\ref{T:everymetric}]
\eqref{it:everyR} The last three equations follow from Theorem~2.5 of \cite{Mi} and the fact that $\g_{\ge} = \z$ is proved in \cite{Gotoh}. For completeness, we supply a proof here. Suppose that $X \notin \z$. Then $Z:=[X,Y] \ne 0$, for some $Y \in \g$. Moreover, $X, Y$ and $Z$ are linearly independent by Lemma~\ref{L:general}\eqref{IT:br}. By \cite[Theorem~2.5]{Mi}, $\Ric(X) < 0$, for some inner product on $\g$, so $X \notin \g_{\ge}$. It follows that $\g_{\ge} \subset \z$. The opposite inclusion is immediate from \eqref{E:RicX}, so $\g_{\ge} = \z$. Again, by \eqref{E:RicX}, a vector $X \in \z$ does not belong to $\g_>$ if and only if there exists an inner product such that $X \perp \g'$, and this occurs if and only if $X \notin \g'$ or $X = 0$. Thus $\g_>=(\g' \cap \z) \setminus \{0\}$. Consequently, $\g_0\subset  \{0\} \cup (\z\setminus \g')$. But if $X\in \z\setminus \g'$, then we may choose vectors $Y, Z \in \g$ with $[Y,Z]\not=0$, and take an inner product with $\<X,[Y,Z]\>\not=0$. Then \eqref{E:RicX} gives $\Ric(X)>0$ and so $X\not\in\g_0$. Hence $\g_0=\{0\}$.

We next prove that $\g_{\le}=\{0\}$. Suppose that $Z \in \g_{\le}, \; Z \ne 0$. As $\g$ is nonabelian, $\g'\cap \z\not=\{0\}$. So if $\R Z=\g'$, then $\g'\subset \z$ and hence $Z\in \g'\cap \z$, which would give $Z\in \g_>$, as we saw above. So we may assume that $\R Z \not=\g'$. We claim that there exist $X, Y \in \g$ so that $X,Y,Z$ are linearly independent and  $[X,Y]  \notin \R Z$. Indeed, as  $\R Z \not=\g'$, there exist $X, Y \in \g$ with $[X,Y] \notin \R Z$. If  $Z \in \Span(X,Y)$, consider the subalgebra $\h$ of $\g$ generated by $X,Y$. Note that $\h$ is a nilpotent algebra and $[X, Y] \notin \Span(X,Y)$ by Lemma~\ref{L:general}\eqref{IT:br}. If $Z=aX+bY$, define $X'=X+a[X,Y]$ and $Y'=Y+b[X,Y]$. Note that $X',Y',[X,Y]$ are linearly independent and by construction, $Z \notin \Span(X',Y')$. Furthermore, $[X',Y']$ is nonzero, as $[X',Y']=[X,Y]$ modulo $[\h,\h']$. So, as $Z \notin \h'$, we have $[X',Y']  \notin \R Z$, as claimed. Thus, by replacing $X,Y$ by $X',Y'$, we obtain that $X,Y,Z$ are linearly independent and  $[X,Y]  \notin \R Z$, as claimed.
Now choose an inner product $\<\cdot, \cdot \>$ on $\g$ such that $X, Y, Z$ are orthonormal and $\<Z,[X,Y]\> \ne 0$. Let $\{e_1, \dots, e_n\}$ be an orthonormal basis for $\g$ relative to $\<\cdot, \cdot \>$ such that $e_1=Z$, $e_{n-1} = X, e_n = Y$. Consider a one-parameter deformation $g_t$ of the inner product $\<\cdot, \cdot \>$ defined by $g_t(U,V)=\<e^{Dt}U,V\>$, where $D=\mathrm{diag}(\lambda_1, \dots, \lambda_n$) is a diagonal matrix relative to the basis $\{e_i\}$. Then the basis $\{E_i=e^{-\lambda_it/2}e_i\}$ is orthonormal for $g_t$ and from \eqref{E:RicX}, the Ricci curvature $\Ric_t(Z)$ of the inner product $g_t$ in the direction $Z$ has the following form:
\begin{align*}
    \Ric_t(Z)&=\Ric_t(e_1)=\frac14\sum_{i,j} g_t(e_1,[E_i,E_j])^2-\frac12\sum_{i}g_t([e_1,E_i],[e_1,E_i]) \\
     &=\frac14\sum_{i,j} g_t(e_1,[E_i,E_j])^2 -\frac12\sum_{i,j} g_t(E_j,[e_1,E_i])^2\\
   &=\frac14\sum_{i,j} e^{(2\lambda_1-\lambda_i-\lambda_j)t} \<e_1,[e_i,e_j]\>^2 -\frac12\sum_{i,j}e^{(\lambda_j-\lambda_i)t} \<e_j,[e_1,e_i]\>^2.
\end{align*}
Now choose the $\lambda_i$'s in such a way that $\lambda_1 > \lambda_2 \ge \dots \ge \lambda_{n-2} > \lambda_{n-1} > \lambda_n$. Then $2\lambda_1-\lambda_{n-1}-\lambda_n > 2\lambda_1-\lambda_i-\lambda_j$, for any $i\ne j, \; \{i,j\}\ne\{n-1, n\}$, and $2\lambda_1-\lambda_{n-1}-\lambda_n >$ $\lambda_j-\lambda_i$, for any $i,j$. It follows that
\begin{equation*}
\lim_{t \to \infty }e^{-(2\lambda_1-\lambda_{n-1}-\lambda_n)t}\Ric_t(Z) =\frac12 \<e_1,[e_{n-1},e_n]\>^2=\frac12 \<Z,[X,Y]\>^2 > 0,
\end{equation*}
which contradicts the fact that $Z \in \g_{\le}$.
Therefore $\g_{\le}=\{0\}$. The fact that $\g_< =\varnothing$ now follows immediately.

\medskip

\eqref{it:everys} Let $\{e_i: i=1, \dots, n\}$ be an orthonormal basis for $\g$ relative to $\<\cdot, \cdot \>$. As in part (a), consider a one-parameter deformation $g_t$ of the inner product $\<\cdot, \cdot \>$ defined by $g_t(U,V)=\<e^{Dt}U,V\>$, where $D=\mathrm{diag}(\lambda_1, \dots, \lambda_n$) is a diagonal matrix relative to the basis $\{e_i\}$. Then by a direct calculation from \eqref{E:sect}, for the inner product $g_t$ we get:
\begin{equation}\label{E:secdef}
\begin{gathered}
    K_t(X,Y)=\sum\nolimits_{i,j,k} e^{(\lambda_j+\lambda_k-\lambda_i)t} \Psi_{ijk}(X,Y)+\sum\nolimits_i e^{\lambda_it} \Phi_{i}(X,Y), \quad \text{where}\\
    \Psi_{ijk}(X,Y)\!=\!\tfrac14(\mu_{ij} (X, Y)+\mu_{ij} (Y, X)) (\mu_{ik} (X, Y)+\mu_{ik} (Y, X)) \!-\!\mu_{ij} (X, X) \mu_{ik} (Y, Y),\\
    \mu_{ij} (X, Y) = \<X,e_j\>\<e_j,[e_i,Y]\>,\\
    \Phi_{i}(X,Y)= -\tfrac34 \<[X, Y],e_i\>^2 -\tfrac12 \<Y,e_i\>\<[X, [X, Y]], e_i\> -\tfrac12 \<X,e_i\>\<[Y, [Y, X]], e_i\>.
\end{gathered}
\end{equation}

First suppose that  $\sigma=\Span(X,Y) \in \G_{\le}$. Taking $\lambda_1 > \lambda_2 \ge \dots \ge \lambda_{n-1} > \lambda_n$ we find that the maximal exponent in the expression for $K_t(X,Y)$ in \eqref{E:secdef} when $t \to \infty$ is $(\lambda_1+\lambda_1-\lambda_n)t$, so we necessarily have $\Psi_{n11} \le 0$. As $\Psi_{n11}=\frac14 (\mu_{n1} (X, Y)-\mu_{n1} (Y, X))^2$, this is only possible when
\begin{equation*}
    \<X,e_1\>\<e_1,[e_n,Y]\>=\<Y,e_1\>\<e_1,[e_n,X]\>,
\end{equation*}
which must be satisfied for any choice of the inner product $\ip$ and the orthonormal basis $\{e_i\}$. Now for a fixed inner product $\ip$ choose $e_n$ in such a way that $\rk(X,Y,e_n)=3$ (this is always possible as $\g$ is nonabelian, so $n=\dim \g \ge 3$), and then take $e_1$ to be an arbitrary unit vector orthogonal to $X$ and to $e_n$, but not orthogonal to $Y$. Then we get $e_1 \perp [e_n,X]$, which implies $[e_n,X] \in \Span(X,e_n)$ by continuity. But then $[e_n,X]=0$ by Lemma~\ref{L:general}\eqref{IT:br}, so by continuity $X \in \z$. Similarly $Y \in \z$ and then by \eqref{E:sect}, $K(X,Y)=0$ for any inner product $\ip$ on $\g$. It follows that $\G_{\le}=\G_{0}=G(2,\z)$, and hence $\G_< = \varnothing$.

\smallskip

Now suppose that $\sigma=\Span(X,Y) \in \G_{\ge}$. In \eqref{E:secdef}, choose the inner product $\ip$ and the orthonormal basis $\{e_i\}$ for $\g$ in such a way that $e_1 \perp \sigma$ and then take $\lambda_1=1, \; \lambda_2=\dots=\lambda_n=0$. Then $\mu_{i1}(X,Y)=\mu_{i1}(Y,X)=\mu_{i1}(X,X)=\mu_{i1}(Y,Y)=0$, so $\Psi_{ij1}=\Psi_{i1j}=0$ for all $i,j=1, \dots,n$, therefore the maximal (potentially nonzero) exponent in the expression for $K_t(X,Y)$ in \eqref{E:secdef} when $t \to \infty$ is $\lambda_1t=t$. Hence we must necessarily have $\Phi_{1} \ge 0$, which implies $\<[X, Y],e_1\>=0$, and then $[X,Y] \in \Span(X,Y)$, so $[X,Y]=0$ by Lemma~\ref{L:general}\eqref{IT:br}.

Furthermore, if $\sigma \cap \z$ is nonzero, then $\sigma \in \G_{\ge}$ by \cite[Corollary~1.3]{Mi}. Suppose that $\sigma \cap \z = \{0\}$. Then $n \ge 4$, as the only nonabelian three-dimensional nilpotent Lie algebra is the Heisenberg algebra, for which any abelian two-dimensional subalgebra $\sigma$ contains the centre. What is more, $\sigma$ is not an ideal of $\g$, as otherwise by Lie's Theorem, the adjoint representation of $\g$ on $\sigma$ would have had a nonzero kernel, which would then be spanned by a vector from $\z$. It follows that there exist $e \in \g$ and a $Y \in \sigma$ such that $\rk(\sigma \cup [e,Y])=3$. Note that this condition is open (so that for almost all $Y \in \sigma$ there exists an open, dense set of $e \in \g$ for which it holds) and that it implies $\rk(e,Y,[e,Y])=3$ and $X \notin \Span(e,Y,[e,Y])$ (where $X$ is an arbitrary vector such that $\sigma=\Span(X,Y)$). The former easily follows from Lemma~\ref{L:general}\eqref{IT:br}; to show the latter we assume that $X=\alpha e + \beta Y + \gamma [e,Y]$, where necessarily $\alpha \ne 0$. Then $0=[X,Y]=\alpha [e,Y] + \gamma [[e,Y],Y]$, which by Lemma~\ref{L:general}\eqref{IT:2br} implies $[[e, Y],Y]=0$, and hence $[e,Y]=0$. We now choose in \eqref{E:secdef} the inner product $\ip$ and the orthonormal basis $\{e_i\}$ for $\g$ in such a way that $e_n =e, \; e_1 \perp \Span(e, Y, [e,Y])$ and $e_1 \not\perp X$ (this is always possible since $n \ge 4$ and $X \notin \Span(e,Y,[e,Y])$) and then take $\lambda_1=10, \; \lambda_2=9, \; \lambda_3=\dots=\lambda_{n-1}=2, \; \lambda_n=0$. Then $\Psi_{n11}=0$, so the maximal (potentially nonzero) exponent in the expression for $K_t(X,Y)$ in \eqref{E:secdef} when $t \to \infty$ is $(\lambda_1+\lambda_2-\lambda_n)t=19t$. Hence we must necessarily have $\Psi_{n12}+ \Psi_{n21} \ge 0$, which gives
\begin{equation*}
    \<X,e_1\>\<e_1,[e_n,X]\> \<Y,e_2\>\<e_2,[e_n,Y]\> \le 0.
\end{equation*}
For this to hold we either have to have $\<e_1,[e_n,X]\>=0$, or otherwise, as the choice of $e_2$ in $\Span(e_1,e_n)^\perp$ was arbitrary, the projections of the vectors $Y$ and $[e_n,Y]$ to $\Span(e_1,e_n)^\perp$ must be collinear. The second possibility quickly leads to a contradiction, as by our choice, $e_n=e$ and $e_1 \perp Y, [e,Y]$, so $[e,Y]-\<[e,Y],e\>e \parallel Y-\<Y,e\>e$, which contradicts the fact that $\rk(e,Y,[e,Y])=3$ established above. Thus $\<e_1,[e,X]\>=0$ for all $e_1 \perp \Span(e, Y, [e,Y])$ (the condition $e_1 \not\perp X$ can be dropped by continuity), therefore $[e,X] \in \Span(e, Y, [e,Y])$. It follows that $[e,X] = \alpha e + \beta Y + \gamma [e,Y]$ for some $\alpha, \beta, \gamma \in \R$. Let $l \ge 1$ be such that $\ad^l_Y \ne 0$, but $\ad^{l+1}_Y = 0$. Acting on the both sides of the last equation by $\ad^l_Y$ and using the fact that $[X,Y]=0$ (so that $\ad_Y$ and $\ad_X$ commute) we get $\ad_X(\ad^l_Ye) = \alpha (\ad^l_Ye)$, so $\alpha=0$, which gives $[e,X] = \beta Y + \gamma [e,Y]$. As the condition defining $e$ was open we can choose $X$ and $Y$ spanning $\sigma$ such that for almost all $e \in \g$ we get $[e,X] = \beta Y + \gamma [e,Y]$ and $[e,Y] = \delta X + \xi [e,X]$ for some $\beta, \gamma, \delta, \xi \in \R$ which depend on $X,Y,e$. But then $[e,Y] = \delta X + \xi (\beta Y + \gamma [e,Y])$, so $\delta=0$ as $\rk(\sigma \cup [e,Y])=3$. It follows that $[e,Y] = \xi [e,X]$, so by continuity, $[Z,Y] \parallel [Z,X]$ for all $X, Y \in \sigma$ and all $Z \in \g$.

Thus a necessary condition for $\sigma \in \G_{\ge}$ is that $\sigma=\Span(X,Y)$, with $[X,Y]=0$ and $[Z,Y] \parallel [Z,X]$, for any $Z \in \g$ (alternatively, for any $Z \in \g$ there exists $X \in \sigma$ such that $[X,Z]=0$). Note that if $\sigma$ has a nonzero intersection with $\z$, this condition is also satisfied. To check that this condition is also sufficient, we let $\ip$ be an arbitrary inner product on $\g$ and $\{e_i\}$ be an orthonormal basis. As $[X,Y]=0$ we get from \eqref{E:sect} that $K(X,Y)=\|\bu(X,Y)\|^2-\<\bu(X,X), \bu(Y,Y)\>$. Denote $u_i=[e_i, X], \; v_i=[e_i, Y]$. Then $\|\bu(X,Y)\|^2-\<\bu(X,X), \bu(Y,Y)\> = \sum_i(\frac14(\<X,v_i\>+\<Y,u_i\>)^2-\<X,u_i\>\<Y,v_i\>)$. But as $u_i \parallel v_i$ we have $\<X,u_i\>\<Y,v_i\> =\<X,v_i\>\<Y,u_i\>$, so %$K(X,Y)= \sum_i(\frac14(\<X,v_i\>+\<Y,u_i\>)^2-\<X,v_i\>\<Y,u_i\>=\sum_i(\frac14(\<X,v_i\>-\<Y,u_i\>)^2$, as required.
\begin{equation}\label{E:Knonneg}
K(X,Y)= \frac14\sum\nolimits_i(\<X,v_i\>-\<Y,u_i\>)^2 = \frac14\sum\nolimits_i(\<X,[e_i, Y]\>-\<Y,[e_i, X]\>)^2,
\end{equation}
as required.

\smallskip

Before examining $\G_>$, we pause to further clarify the nature of $\G_\ge$. We will also need the following lemma in the subsequent consideration of $\G_>$.

\begin{lemma}\label{L:Gge}
The set $\G_{\ge}$ can be represented as
  $\G_{\ge}=\G_1 \cup \G_2$, where
  \begin{itemize}
    \item $\G_1$ is the set of all the two-planes having a nontrivial intersection with $\z$,
    \item $\G_2$ is the set of all the two-planes $\sigma$ with the following property: there exists a three-dimensional
    abelian ideal $\a_3 \supset \sigma$ such that $\dim[\g,\a_3]=1$.
  \end{itemize}
\end{lemma}
\begin{proof}
We have shown that a two-plane $\sigma=\Span(X,Y)$ is in $\G_{\ge}$ if and only if $[X,Y]=0$ and for every $Z \in \g$ we have $\ad_XZ \parallel \ad_YZ$.

Clearly, $\G_1 \subset \G_{\ge}$. Let $\sigma\in \G_{\ge}$. If $\ad_X=0$, then $\sigma \in \G_1$. Otherwise, $\ad_XZ_1 =P \ne 0$ for some $Z_1 \in \g$ and we can take $Y \in \sigma$ such that $\ad_YZ_1=0$. Consider $Z_2\in\g$ with $\ad_XZ_2=P_2\not=0$. Then $\ad_YZ_2=aP_2$ for some $a\in\R$, and so
\begin{equation*}
  \ad_X(Z_1+Z_2)=P+P_2, \qquad \ad_Y(Z_1+Z_2)=aP_2.
\end{equation*}
Then either $P_2 \parallel P$ or $a=0$, and in the latter case, $\ad_YZ_2=0$. If for some vector $Z_2$ we have $P_2 \nparallel P$, then $\ad_XZ \nparallel P$, for almost all $Z \in \g$, and hence $\ad_YZ=0$, so $Y\in\z$ and $\sigma\in \G_1$. Otherwise, we have $\ad_XZ , \ad_YZ\parallel P$, for all $Z \in \g$. Therefore there exist one-forms $\theta_1, \theta_2 \in \g^*$ such that
\begin{equation}\label{E:theta}
    [X,Z]=\theta_1(Z) P, \quad [Y,Z]=\theta_2(Z) P, \quad \text{for all} \; Z \in \g.
\end{equation}
From the first equation of \eqref{E:theta} it follows that $[[X,Z],U]=\theta_1(Z) [P,U]$, so by the Jacobi identity, $\theta_1(Z) [P,U]-\theta_1(U) [P,Z]+\theta_1([U,Z]) P=0$. Taking $ Z \notin \ker \theta_1$ we get $[P,U]=0$, for all $U \in \ker \theta_1$, by Lemma~\ref{L:general}\eqref{IT:br}. Similarly, from the second equation of \eqref{E:theta}, $[P,U]=0$, for all $U \in \ker \theta_2$. Note that for $\sigma \in \G_{\ge} \setminus \G_1$, we must have $\theta_1 \nparallel \theta_2$ in \eqref{E:theta}. This implies that $P \in \z$. Moreover, $P \notin \sigma$, as otherwise $\sigma \in \G_1$. Now, as $P\in\z$ and by \eqref{E:theta}, the subspace $\a_3=\Span(X,Y,P)$ is a three-dimensional ideal, which is abelian (as $[X,Y]=0$) and satisfies $\dim[\g,\a_3]=1$, so $\sigma\in \G_2$.

Conversely, given any three-dimensional abelian ideal $\a_3$, with $\dim[\g,\a_3]=1$, let $P$ be a nonzero vector from $[\g,\a_3]$. Then $P \in \z$, by Lemma~\ref{L:general}\eqref{IT:br}. Consider a two-plane $\sigma $ in $\a_3$. If $\sigma$ contains $P$, then $\sigma \in \G_1$. Otherwise, equations  \eqref{E:theta} are satisfied (but possibly, with $\theta_1 \parallel \theta_2$), so $\sigma \in \G_{\ge}$. Hence $\G_{\ge}=\G_1 \cup \G_2$.
\end{proof}

\begin{remark}
Concerning the above lemma, note that depending on $\g$, it may, or it may not happen that $\G_{\ge}=\G_1$ (so that $\G_{\ge}$ consists only of the two-planes having a nontrivial intersection with the centre). An example with $\G_{\ge}=\G_1$ is the filiform algebra defined by $[X_i,X_j]=(j-i)X_{i+j}$, for $1 \le i < j, \; i+j \le n$, where $n \ge 3$. Such an algebra does not contain three-dimensional abelian ideals at all. An example with $\G_{\ge}\not=\G_1$ is the Heisenberg algebra defined by
$[X_{2i-1},X_{2i}]=X_{2m+1}, \; i=1, \dots, m, \; n=2m+1 \ge 5$. The two-plane $\sigma=\Span(X_1, X_3)$ (and many others) lies in $\G_{\ge}$, but has a trivial intersection with the centre.
\end{remark}

To find $\G_>$ we use the fact that $\G_> \subset \G_{\ge}$. From Lemma~\ref{L:Gge} and its proof, if $\sigma=\Span(X,Y) \in \G_{\ge}$, then either $\sigma \cap \z \ne 0$ or $[X,Y]=0$ and there exist a nonzero $P \notin \sigma$ and the one-forms $\theta_1, \theta_2 \in \g^*$ such that equation \eqref{E:theta} is satisfied. But in the second case, by \eqref{E:Knonneg} we have $K(X,Y)= \frac14\sum\nolimits_i(\theta_1(e_i)\<X,P\>-\theta_2(e_i)\<Y,P\>)^2$, which vanishes if we choose an inner product in such a way that $\<X,P\>=\<Y,P\>=0$, hence $\sigma \notin \G_>$. In the first case, we can assume that $X \in \z$. By \eqref{E:Knonneg} we get $K(X,Y)=\frac14 \sum_i \<X,[Y,e_i]\>^2$, where $e_i$ is an orthonormal basis for $\g$. This expression is positive, for any choice of the inner product, if and only if $X \in [Y, \g]$. This establishes the theorem for $\G_>$.
\end{proof}

\section{Proof of Theorem~\ref{T:maxmin}}
\label{S:extrema}

As some fragments of the proof of Theorem~\ref{T:maxmin} are rather technical, we start by giving a brief outline. Given a metric nilpotent Lie algebra, there is in general little chance of finding explicitly the vectors on which the Ricci curvature attains its maximum or minimum. To have some control, we start with an arbitrary inner product $\ip$ on $\g$, and then deform it by $\ip \mapsto \< e^{tD}\cdot , \cdot \>$, where $D$ is a diagonal matrix relative to some orthonormal basis for $\ip$; geometrically, we travel along a geodesic in the space of inner products on $\g$, which can be identified with a noncompact Riemannian symmetric space $\R^+ \times \mathrm{SL}(n)/\mathrm{SO}(n)$. The Ricci tensor of the deformed inner product, after scaling, has a limit when $t \to \infty$, for which the eigenvectors with the greatest and the smallest eigenvalues can be found explicitly. Moreover, the projective classes of these eigenvectors belong to $\overline{M}$ and $\overline{m}$ respectively, provided the corresponding eigenspaces are one-dimensional. These computations (done in Lemma~\ref{L:mM} for different choices of $D$) provide us with a supply of elements from $\overline{M}$ and $\overline{m}$ rich enough to prove assertions \eqref{IT:max}\eqref{IT:2step}, \eqref{IT:max}\eqref{IT:codim1} and \eqref{IT:min} of the theorem.

To prove assertion \eqref{IT:max}\eqref{IT:other} we first consider the ``generic nilpotent Lie algebras", which we define by requiring that at least one of the equalities \eqref{E:rk5} or \eqref{E:rk7} below is satisfied for some pair (respectively, triple) of elements from $\g$. The ``non-generic algebras" are classified in Lemma~\ref{L:structure}: they are either two-step nilpotent, or are one-dimensional extensions (central or by a nilpotent derivation) of two-step nilpotent ones. We then reduce the non-generic case to considering a small list of low-dimensional algebras: namely, of five- and six-dimensional nilpotent Lie algebras, which are one-dimensional extensions of two-step nilpotent algebras.

\begin{proof}[Proof of Theorem~\ref{T:maxmin}]
If $\g$ is abelian, then the Ricci curvature of any metric is identically zero by \eqref{E:RicX}, so $\overline{M} = \overline{m} = \mathbb{P}\g$. We will assume for the rest of the proof that $\g$ is nonabelian.

We use the following notation. Given elements $u_i \in \g$, denote $u_{ij}=[u_i,u_j]$ and $u_{ijk}=[u_i,[u_j,u_k]]$. For $k \ge 2, \; \g^k$ is the $k$-th Cartesian power of $\g$, the $k$-fold Cartesian product of $\g$ with itself. For a triple $(X_1,X_2, X_3)\in \g^3$, denote $\mathcal{L}(X_1,X_2,X_3) = \Span(X_1,X_2,X_3,X_{12},X_{23},X_{13})$.

The proof is based on the following key technical lemma, the proof of which will be given in Section~\ref{S:lemma5}. Recall that for a linear space $L$, we denote $\mathbb{P}L$ the projective space of $L$, and $\pi: L \setminus \{0\} \to \mathbb{P}L$ is the natural projection.

\begin{lemma}\label{L:mM}
Let $(\g, \ip)$ be a nonabelian metric nilpotent Lie algebra.
\begin{enumerate}[{\rm (a)}]
  \item \label{IT:e1u2}
  For any orthonormal vectors $e, u_1, u_2 \in \g$ such that the vector
  \begin{equation*}
  T=2\<u_{12},e\>u_{12}+\<u_{212},e\>u_1-\<u_{112},e\>u_2
  \end{equation*}
  is nonzero, $\pi(T) \in \overline{M}$.

%  For vectors $ u_1, u_2 \in \g$ such that $u_{212}\not=0$ and $\<u_{112},u_{212}\>=0$, we have $\pi(u_1) \in \overline{M}$.
  \item \label{IT:eu}
  Let $e_1, e_2, u_1, u_2, u_3 \in \g$ be orthonormal vectors such that
  \begin{equation*}
  \<e_1, u_{13}\>=\<e_1, u_{23}\>=\<e_2, u_{12}\>=\<e_2, u_{23}\>=0, \; \<e_1, u_{12}\>=a \ne 0, \; \<e_2, u_{13}\>=b \ne 0.
  \end{equation*}
  \begin{enumerate}[{\rm (i)}]
    \item \label{IT:eumin} Then $\pi(u_1) \in \overline{m}$.
    \item \label{IT:e2u3}
    Suppose additionally that $|a| > |b|$. Let
    \begin{align*}
    T_1&=2(b\<e_1,u_{212}\>+a\<e_2,u_{312}\>)u_1-3b\<e_1,u_{112}\>u_2-3a\<e_2,u_{112}\>u_3 + 6abu_{12},\\
    T_2&=\tfrac{1}{2a^2+b^2}(a\<e_1,u_{212}\>+b\<e_2,u_{312}\>)u_1-\tfrac{1}{2a}\<e_1,u_{112}\>u_2-\tfrac{b}{a^2+b^2}\<e_2,u_{112}\>u_3 + u_{12}.
    \end{align*}
 Then for  $i=1,2$, we have $\pi(T_i) \in \overline{M}$, provided $T_i \ne 0$.
  \end{enumerate}
  \item \label{IT:2s}
  Suppose $\g$ is two-step nilpotent. For any unit vector $e \in \g'$ and any orthonormal basis $\{u_1, \dots, u_q\}$ for $(\g')^\perp$ for which
  $T=\sum\nolimits_{i,j=1}^q\<e,u_{ij}\>u_{ij} \ne 0$, we have $\pi(T) \in \overline{M}$.
  \end{enumerate}
\end{lemma}

We will also make use of the following lemma whose proof is given in Section~\ref{S:lemmata}.

\begin{lemma}\label{L:notL}
Let $\g$ be a nonabelian nilpotent Lie algebra. Suppose for all $(X_1,X_2,X_3) \in \g^3$, $\dim\mathcal{L}(X_1,X_2,X_3) \le 4$. Then $\g$ is either the direct product of a Heisenberg algebra and a \emph{(}possibly trivial\emph{)}  abelian ideal, or $\g$ is the four-dimensional filiform algebra $\Span(X,Y,Z,W)$ given by the relations $[W,X]=Y, \; [W,Y]=Z$. In particular, $\g$ is either two-step nilpotent or has a codimension one abelian ideal.
\end{lemma}

Returning to the proof of Theorem~\ref{T:maxmin}, we consider the various parts:

\medskip

\eqref{IT:max}\eqref{IT:2step}
If $\g$ is two-step nilpotent, then by (\ref{E:RicXY}, \ref{E:RicX}), for any inner product $\ip$ on $\g$ and for any $Y \in (\g')^\perp, \; X \in \g'$, we have $\Ric(Y) \le 0, \; \Ric(X,Y)=0$ and $\Ric(X) \ge 0$, and furthermore, $\Ric(X) > 0$ for some $X \in \g'$, as $\g$ is nonabelian. It follows that $\g'$ and $(\g')^\perp$ are invariant subspaces of the linear map $\ric$, and the maximum of $\Ric$ on the unit sphere of $(\g, \ip)$ is attained on some vector from $\g'$, so $\overline{M} \subset \mathbb{P}\g'$. To prove the converse, fix an inner product $\ip$ on $\g$ and an orthonormal basis $\{u_1,\dots,u_q\}$ for $(\g')^\perp$. Define $\psi \in \End(\g')$ by $\psi(e)=\sum\nolimits_{i,j=1}^q\<e,u_{ij}\>u_{ij}$, for $e \in \g'$. By Lemma~\ref{L:mM}\eqref{IT:2s}, $\pi(\psi(e)) \in \overline{M}$, if $\psi(e) \ne 0$. Now, as $\<\psi(e),e\>=\sum\nolimits_{i,j=1}^q\<e,u_{ij}\>^2$ and as $\g'$ is spanned by the $u_{ij}$'s, the vector $\psi(e)$ is nonzero if $e \ne 0$. So $\ker\psi=0$, and hence $\psi$ is surjective. Then $\overline{M} \supset \mathbb{P}(\psi(\g'))= \mathbb{P}\g'$, as required.

\medskip

\eqref{IT:max}\eqref{IT:codim1}
Suppose $\g=\R c \oplus \a$, where $\a$ is a codimension one abelian ideal. The fact that $\g$ is not two-step nilpotent means that $\ad_c^2(\a) \ne 0$. We first prove that $\overline{M} \supset \mathbb{P}\a$. Take a vector $u_1 \in \a$ such that $v:=[c,[c,u_1]] \ne 0$ (such vectors $u_1$ form an open, dense subset of $\a$).  By Lemma~\ref{L:general}, the vectors $c,u_1, [c,u_1], v$ are linearly independent. Choose an inner product on $\g$ for which $c,u_1,  [c,u_1],e$ are orthonormal. In Lemma~\ref{L:mM}\eqref{IT:e1u2}, take $u_2=c$ and $e \perp \Span(c,u_1, [c,u_1])$, but $\<e, v\> \ne 0$. Then $u_{112}=0$, so we have $\pi(u_1) \in \overline{M}$, hence $\overline{M} \supset \mathbb{P}\a$. To prove the converse, suppose $\ip$ is an arbitrary inner product on $\g$, and $c'$ is a unit vector orthogonal to $\a$. By (\ref{E:RicXY}, \ref{E:RicX}), $\Ric(c')<0$ and $\Ric(c',X)=0$, for all $X \in \a$. It follows that any eigenvector of $\ric$ with maximal eigenvalue (which must be positive by \cite[Theorem~2.4]{Mi}) is orthogonal to $c'$ and hence belongs to $\a$. So $\overline{M} \subset \mathbb{P}\a$.

\medskip

\eqref{IT:min}
Suppose there exists $(X_1,X_2,X_3) \in \g^3$ such that $\dim \mathcal{L}(X_1,X_2,X_3) > 4$. Then $\dim \mathcal{L}(X_1,X_2,X_3) > 4$ for almost all $(X_1,X_2,X_3) \in \g^3$. Denote $L_3=\Span(X_1,X_2,X_3)$ (note that $\dim L_3=3$). As $\dim \mathcal{L}(X_1,X_2,X_3)/L_3 \ge 2$, we can choose a two-plane $\sigma \subset \mathcal{L}(X_1,X_2,X_3)$ such that $\dim(L_3 + \sigma) = 5$, and then define an inner product $\ip$ on $\g$ in such a way that $L_3 \perp \sigma$ (we will later specify it further). Consider the linear map $\psi: \Lambda^2 L_3 \to \sigma, \; \psi(X \wedge Y):=\pi_{\sigma}[X,Y]$. %\<e_1,[X,Y]\>e_1+\<e_2,[X,Y]\>e_2
The map $\psi$ is well-defined and surjective, as for $e \in \sigma$ orthogonal to $\psi(\Lambda^2 L_3)$ we would have had $e \perp \mathcal{L}(X_1,X_2,X_3)$. As all the elements of $\Lambda^2 L_3$ are decomposable, $\ker \psi = \R (U \wedge V)$, for some linearly independent $U, V \in L_3$. Denote $L_2=\Span(U,V)$. Take a vector $u_1 \in L_3 \backslash L_2$ and two linearly independent vectors $u_2, u_3 \in L_2$. Then the vectors $e_1=\psi(u_1\wedge u_2)$ and $e_2=\psi(u_1\wedge u_3)$ are linearly independent and span $\sigma$. We now specify the inner product $\ip$ further by requiring the vectors $u_i, e_j$ to be orthonormal. Then the vectors $u_i,e_j$ satisfy the assumptions of Lemma~\ref{L:mM}\eqref{IT:eu}, so by Lemma~\ref{L:mM}\eqref{IT:eu}\eqref{IT:eumin}, $\pi(u_1) \in \overline{m}$. Since $u_1 \in L_3\backslash L_2$ was arbitrary, it follows that $\mathbb{P}L_3 \subset \overline{m}$. As this is satisfied for almost all $L_3=\Span(X_1,X_2,X_3)$, we get $\overline{m}=\mathbb{P}\g$, as required.

Now suppose that $\dim \mathcal{L}(X_1,X_2,X_3) \le 4$ for any triple of vectors $(X_1,X_2,X_3) \in \g^3$. By Lemma~\ref{L:notL}, this could only happen when $\g$ is either the direct product of a Heisenberg algebra and a (possibly trivial) abelian algebra, or is the four-dimensional filiform algebra $\Span(W,X,Y,Z)$ given by $[W,X]=Y, \; [W,Y]=Z$.

In the second case, choose the inner product such that the vectors $E_1=W+aX+bY+cZ, \; E_2=X, \; E_3=Y$, $E_4=Z$ are orthonormal (with arbitrary $a,b,c \in \R$). Then $[E_1,E_2]=E_3, \; [E_1, E_3]=E_4$ and a direct computation shows that, relative to the basis $\{E_i\}$, the Ricci operator is diagonal, with the diagonal entries $-1, -\frac12, 0, \frac12$, respectively. It follows that $m \ni \pi(E_1)=\pi(W+aX+bY+cZ)$, so $\overline{m}=\mathbb{P}\g$.

If $\g$ is the direct product of a Heisenberg algebra given by $[X_{2i-1},X_{2i}]=X_{2l+1}, \; i=1, \dots, l$, and an abelian algebra $\a=\Span(X_{2l+2},\dots,X_n)$, choose the inner product in such a way that the vectors
\begin{equation*}
E_i=\begin{cases}
X_i+Z_i+a_iX_{2l+1}&:\ \text{for}\ i=1, \dots ,2l\\
X_i&:\ \text{for}\ i=2l+1, \dots ,n,
\end{cases}
\end{equation*}
are orthonormal, where $Z_i \in \a$ and $a_i \in \R$ are arbitrary. Then the relations for the $E_i$'s are the same as those for the $X_i$'s and a direct computation shows that, relative to the basis $\{E_i\}$, the Ricci operator is diagonal, with $\Ric(E_i)=-\frac12$ for $i=1, \dots, 2l$, $\Ric(E_{2l+1}) =\frac{l}2$, and $\Ric(E_j)=0$ for $j=2l+2, \dots, n$. It follows that every nontrivial linear combination of $E_1, \dots, E_{2l}$ is an eigenvector of $\ric$ with the smallest eigenvalue. Choosing $Z_i$ and $a_i$ arbitrarily we obtain $\overline{m}=\mathbb{P}\g$.

\medskip

\eqref{IT:max}\eqref{IT:other} ``generic case".
We show that $\overline{M} =\mathbb{P}\g$ for every algebra $\g$ satisfying one of the open conditions \eqref{E:rk5} or \eqref{E:rk7} below.

Suppose that there exist vectors $X_1, X_2 \in \g$ such that
\begin{equation}\label{E:rk5}
\rk(X_1,X_2, X_{12}, X_{112},X_{212})=5.
\end{equation}
If condition~\eqref{E:rk5} is satisfied for at least one pair $(X_1, X_2) \in \g^2$, then it is satisfied for almost all pairs $(X_1, X_2) \in \g^2$. Choose one such pair and define an inner product on $\g$ in such a way that the five vectors from \eqref{E:rk5} are orthonormal. By Lemma~\ref{L:mM}\eqref{IT:e1u2} with $u_1=X_1, \, u_2=X_2, \, e=X_{212}$, we have $\pi(X_1)\in \overline{M}$, which implies $\overline{M} =\mathbb{P}\g$.

Suppose that there exist vectors $X_1, X_2, X_3 \in \g$ such that
\begin{equation}\label{E:rk7}
\rk(X_1,X_2, X_3,X_{12}, X_{13},X_{23},X_{312})=7.
\end{equation}
As before, if \eqref{E:rk7} is satisfied for at least one triple $(X_1, X_2, X_3) \in \g^3$, then it is satisfied for almost all of them. Choose one such triple and define an inner product on $\g$ in such a way that the seven vectors from \eqref{E:rk7} are orthonormal. Then the vectors $u_i=X_i, \; e_1=X_{12}$ and $e_2=X_{13} \cos \tau+X_{312}\sin \tau$, $\tau \in(0,\frac\pi2)$, are orthonormal and satisfy the hypothesis of  Lemma~\ref{L:mM}\eqref{IT:eu}\eqref{IT:e2u3}, with $a=1, \; b=\cos \tau$. By that assertion, if the vector $T_1=2(\<X_{12},X_{212}\>\cos \tau +\sin \tau) X_1 - 3\<X_{12},X_{112}\> \cos \tau X_2- 3(\<X_{13},X_{112}\>\cos \tau + \<X_{312},X_{112}\>\sin \tau) X_3+6\cos \tau X_{12}$ is nonzero, then $\pi(T_1) \in \overline{M}$. Taking the limit as $\tau \to \frac\pi2$ we obtain
\begin{equation}\label{E:312112}
\pi(2 X_1 -3\<X_{312},X_{112}\> X_3)\in \overline{M}.
\end{equation}
Now, if $X_{112}$ does not belong to the span of the seven vectors from \eqref{E:rk7}, we could additionally assume that the inner product is chosen in such a way that $X_{112}$ is orthogonal to them. Then $\pi(X_1)\in \overline{M}$, by \eqref{E:312112}. If $X_{112}$ belongs to the span of the seven vectors from \eqref{E:rk7}, then $X_{112}-\mu X_{312} \in \mathcal{L}(X_1,X_2,X_3)$, for some $\mu \in \R$, and so $\<X_{312},X_{112}\>=\mu$. Thus $\pi(2 X_1 -3 \mu X_3)\in \overline{M}$. If $\mu=0$ we have  $\pi(X_1)\in \overline{M}$. Assume $\mu\not=0$, and replace the triple $X_1,X_2,X_3$ by the triple $X_1,X_2, X_3(t)=X_3+tX_1$. This does not violate condition \eqref{E:rk7} provided $t\not=-\mu^{-1}$. Set $\mu(t):=(1+\mu t)^{-1}\mu$. Then
\begin{equation*}
X_{112}-\mu(t) [X_3(t),[X_1,X_2]] = \frac{1}{1+\mu t}(X_{112} -\mu X_{312}) \in \mathcal{L}(X_1,X_2,X_3).
\end{equation*}
Thus $\mu(t)$ plays the same role for $X_1,X_2, X_3(t)$ as $\mu$ did for $X_1,X_2, X_3$. It follows that
\begin{equation*}
\overline{M}\ni \pi(2 X_1 -3 \mu(t) X_3(t))=\pi\Big(\frac{2-\mu t}{1+\mu t} X_1 - \frac{3\mu}{1+\mu t} X_3\Big).
\end{equation*}
Taking the limit as $t$ tends to infinity, we obtain $\pi(X_1)\in \overline{M}$. So for almost all triples $(X_1, X_2, X_3) \in \g^3$, we have $\pi(X_1)\in \overline{M}$, which implies $\overline{M} =\mathbb{P}\g$.

\medskip

\eqref{IT:max}\eqref{IT:other} ``non-generic case".
To complete the proof of the theorem, it remains to consider the algebras $\g$ for which both conditions \eqref{E:rk5} and \eqref{E:rk7} are violated, but which are not two-step nilpotent and do not contain a codimension one abelian ideal. As one may expect, these conditions are very restrictive, which is confirmed by the following lemma whose proof we postpone till Section~\ref{S:lemmata}.

\begin{lemma}\label{L:structure}
Let $\g$ be a nilpotent Lie algebra, which is nonabelian and not two-step nilpotent. Suppose that both conditions \eqref{E:rk5} and \eqref{E:rk7} are violated, for all pairs \emph{(}respectively, triples\emph{)} of vectors from $\g$. Then $\g$ belongs to one of the classes \eqref{IT:b} or \eqref{IT:c} below.
  \begin{enumerate}[{\rm (A)}]
    \item \label{IT:b}
    $\g$ is a one-dimensional central extension of a two-step nilpotent Lie algebra $\h$ by a cocycle $\omega$ with the following property: for almost all $X \in \h$, there exists $Y \in \h$ with $\omega(X,[X,Y]_\h)=0$ and $\omega(Y,[X,Y]_\h) \ne 0$.
    \item \label{IT:c}
    $\g$ is a one-dimensional extension of a two-step nilpotent ideal $\h \subset \g$ by a nilpotent derivation $D$ of $\h$ such that $[DX,X]=0$, for all $X \in \h$.
  \end{enumerate}
Furthermore, suppose that $\g$ belongs to class \eqref{IT:c}. Then there exists $N \in \{5, 6\}$ such that for almost all $(X_1,X_2,X_3) \in \g^3$, the subspace $L:=\mathcal{L}(X_1,X_2,X_3)$ is a subalgebra of $\g$ of dimension $N$ isomorphic to one of the following algebras:
\begin{enumerate}[{\rm (a)}]
  \item \label{IT:five}
  If $N=5$, then $L\cong\Span(c,X,Y,Z,A)$ defined by the relations
  \begin{equation*}
  [c,X]=A, \; [c,A]=Z, \; [X,Y]=Z.
  \end{equation*}

  \item \label{IT:six}
  If $N=6$, then $L\cong\Span(c,X,Y,Z,A_1,A_2)$ defined by one of the following sets of relations:
  \begin{gather}\label{E:dimsix1}
  [c,X]=A_1, \; [c,A_1]=A_2, \; [X,Y]=Z;\\
  [c,X]=A_1, \; [c,Y]=A_2, \; [c,A_1]=Z, \; [X,Y]=Z; \label{E:dimsix2}\\
  [c,X]=A_1, \; [c,A_1]=A_2, \; [c,A_2]=Z, \; [X,Y]=Z.\label{E:dimsix3}
  \end{gather}
\end{enumerate}
\end{lemma}
% do we include abelian in two-step?

We now separately examine the Lie algebras of classes \eqref{IT:b} and \eqref{IT:c}.

\medskip

\emph{Algebras of class \eqref{IT:b}}. Let $\g$ be a one-dimensional central extension of a two-step nilpotent Lie algebra $\h$ by a cocycle $\omega$, so
that $\g=\h \oplus \R c$, with the Lie bracket defined by $[c,\g]=0$ and $[X,Y]=[X,Y]_\h+\omega(X,Y)c$, for $X,Y \in\h$, and furthermore, for almost all $X \in \h$, there exists $Y \in \h$ with $[X,[X,Y]]=0, \; [Y,[X,Y]]= \gamma c, \; \gamma:=\omega(Y,[X,Y]_\h)\ne 0$. Choose any two such $X, Y$ and any $\alpha \in \R$ and denote $u_1=X+\alpha c, \; u_2= Y$. Then $u_{112}=0, \; u_{212}= \gamma c \ne 0$, hence the vectors $u_1, u_2, u_{12}, u_{212}$ are linearly independent by Lemma~\ref{L:general}\eqref{IT:2br}. Choose an inner product $\ip$ for $\g$ in such a way that they are orthonormal and take $e=u_{212}$. Then $\pi(u_1) \in \overline{M}$ by Lemma~\ref{L:mM}\eqref{IT:e1u2}, so $\pi(X+ \alpha c) \in \overline{M}$ for almost all $X \in \h$ and all $\alpha \in \R$. Therefore $\overline{M} =\mathbb{P}\g$, as required.

\medskip

\emph{Algebras of class \eqref{IT:c}}. In the both cases \eqref{IT:five} (when $N=5$) and \eqref{IT:six} (when $N=6$) we will show that $\mathbb{P}L \subset \overline{M}$ for almost all triples $(X_1, X_2, X_3) \in \g^3$, which will then imply $\overline{M}=\mathbb{P}\g$, as required. We consider these two cases separately.

\smallskip

Case~\eqref{IT:five}: $N=5$. For almost all triples $(X_1, X_2, X_3) \in \g^3$, the subspace $L$ is a subalgebra of $\g$ isomorphic to the algebra
$\Span(c,X,Y,Z,A)$ defined by the relations $[c,X]=A, \; [c,A]=Z, \; [X,Y]=Z$. For nonzero reals $\alpha_1,\alpha_2,\alpha_3$, define an inner product on $\g$ such that the vectors
\begin{align*}
u_1&:=6\alpha_1 c+\alpha_2 X, \quad u_2:= c, \quad u_3:=10\alpha_1 c+\alpha_3 Y, \\
e_1 &:= u_{12}=-\alpha_2 A,
\quad e_2 := \sqrt{2} u_{13}=\sqrt{2} (-10\alpha_1\alpha_2 A+\alpha_2\alpha_3 Z)
\end{align*}
are orthonormal. Using the fact that $u_{23} = 0$, it is easy to verify that the assumptions of Lemma~\ref{L:mM}\eqref{IT:eu}\eqref{IT:e2u3} are satisfied. The choice of vectors $u_1,u_2,u_3,e_1,e_2$ has been made so that, as a direct computation shows, one has
\begin{align*}
T_2&=\tfrac{1}{2a^2+b^2}(a\<e_1,u_{212}\>+b\<e_2,u_{312}\>)u_1 -\tfrac{1}{2a}\<e_1,u_{112}\>u_2-\tfrac{b}{a^2+b^2}\<e_2,u_{112}\>u_3 + u_{12}\\
&=2\alpha_1\alpha_3^{-1} (\alpha_1 c+\alpha_2 X+\alpha_3 Y-\tfrac12 \alpha_1^{-1}\alpha_2 \alpha_3 A).
\end{align*}
Lemma~\ref{L:mM}\eqref{IT:eu}\eqref{IT:e2u3} then gives $\pi(\alpha_1 c+\alpha_2 X+\alpha_3 Y-\tfrac12 \alpha_1^{-1}\alpha_2 \alpha_3 A) =\pi (T_2) \in \overline{M}$. Now for arbitrary reals $\beta_1, \beta_2$, the linear map $\phi$ on $L$ which is the identity on $\Span(X,Y,A,Z)$ and such that $\phi(c) = c + \alpha_1^{-1} (\beta_1 A + \beta_2 Z)$ is an automorphism of $L$. Although $\phi$ may not extend to an automorphism of the entire algebra $\g$, we can replace the basis vectors $u_i, e_i$ defined above by their images under $\phi$ and consider an inner product $\ip_{\phi}$ on $\g$, for which they are orthonormal. Then the assumptions of Lemma~\ref{L:mM}\eqref{IT:eu}\eqref{IT:e2u3} are again satisfied and we obtain that $\overline{M} \ni \pi(\alpha_1 c + \alpha_2 X + \alpha_3 Y + (\beta_1 -\tfrac12 \alpha_1^{-1}\alpha_2 \alpha_3) A +\beta_2 Z)$. It follows that $\pi(L) \subset \overline{M}$, as required.

\smallskip

Case~\eqref{IT:six}: $N=6$. For almost all triples $(X_1, X_2, X_3) \in \g^3$, the subspace $L$ is a subalgebra of $\g$ isomorphic to one of the three algebras (\ref{E:dimsix1},\ref{E:dimsix2},\ref{E:dimsix3}). We treat all three algebras simultaneously. For nonzero reals $\alpha_1,\alpha_2,\alpha_3$, we choose $u_1,u_2,u_3$  as shown in the Table below. We choose the inner product on $\g$ for which the vectors $u_1,u_2,u_3,e_1:=u_{12},e_2:=2u_{13}$ and $u_{23}$ are orthonormal.

\begin{table}[h]
\begin{tabular}{|c|l|l|l|}\hline
Algebra&$u_{1}$&$u_{2}$&$u_{3}$ \\\hline
\eqref{E:dimsix1} and \eqref{E:dimsix3}&$-2\alpha_1 c+\alpha_2 X+\alpha_3 Y$&$ X$ &$ c+ A_1$\\\hline %$\g_{6,2}$ and $\g_{6,19}$
\eqref{E:dimsix2}&$\alpha_1 c-\alpha_2 X$&$ c$&$-6\alpha_1 c-11\alpha_2 X-\alpha_3 Y$\\\hline %$\g_{6,20}$
\end{tabular}
%\medskip
%\caption{Algebras treated using Lemma~\ref{L:mM}\eqref{IT:e2u3}}\label{Table3}
\end{table}

For algebras \eqref{E:dimsix1} and \eqref{E:dimsix3} we obtain
\begin{equation*}
T_1=\alpha_1 c+\alpha_2 X+\alpha_3 Y -3 \alpha_1  A_1 - 3 \alpha_3  Z.
\end{equation*}
while for \eqref{E:dimsix2},
\begin{equation*}
T_2=\tfrac15 \alpha_1 \alpha_3^{-1}(\alpha_1 c+\alpha_2 X+\alpha_3 Y)+\alpha_2 A_1.
\end{equation*}
By Lemma~\ref{L:mM}\eqref{IT:eu}\eqref{IT:e2u3} we have $\pi(T_1), \pi (T_2) \in \overline{M}$. Notice that the subspace $\Span(A_1, A_2, Z)$ is the centre of the codimension one ideal $\ig=\Span(X, Y, A_1, A_2, Z)$, for each of the algebras (\ref{E:dimsix1},\ref{E:dimsix2},\ref{E:dimsix3}). It follows that any linear map $\phi$ on $L$ which is the identity on $\ig$ and such that $\phi(c) = c+ U$, for an arbitrary $U\in \Span(A_1, A_2, Z)$, is an automorphism of $L$. Although
$\phi$ may not extend to an automorphism of the entire algebra $\g$, we can replace the basis vectors $u_i, e_i,u_{23}$ defined above by
their images under $\phi$ and consider an inner product $\ip_{\phi}$ on $\g$, for which they are orthonormal. Then the assumptions of Lemma~\ref{L:mM}\eqref{IT:eu}\eqref{IT:e2u3} are again satisfied. Consequently $\overline{M} \ni \pi(\phi(T_i))=\pi(\alpha_1 c+\alpha_2 X+\alpha_3 Y + U)$ for all $U \in \Span(A_1, A_2, Z)$, and hence $\mathbb{P}L \subset \overline{M}$, as required.
\end{proof}

\section{Proof of Lemma~\ref{L:mM}}
\label{S:lemma5}

Choose an arbitrary inner product $\ip$ on $\g$, with an orthonormal basis $\{e_i\}$. From \eqref{E:RicX}, for $X \in \g$,
\begin{equation}\label{E:ricX}
\begin{split}
    \ric X&=\sum\nolimits_{ijk}\<e_k,[e_i,e_j]\>(\tfrac14 \<e_k,X\>[e_i,e_j]-\tfrac12\<e_k,[e_i,X]\>e_j)\\
    &=\tfrac12\sum\nolimits_{ijk, \, i>j}\<e_k,[e_i,e_j]\>(\<e_k,X\>[e_i,e_j]-\<e_k,[e_i,X]\>e_j+\<e_k,[e_j,X]\>e_i).
\end{split}
\end{equation}
Consider a one-parameter deformation $g_t$ of the inner product defined by $g_t(X,Y)=\<e^{Dt}X,Y\>$, where $D=\mathrm{diag}(\lambda_1, \dots, \lambda_n$) is a diagonal matrix relative to the basis $\{e_i\}$. Then $g_t(e_i, X)=e^{\lambda_it}\<e_i,X\>$ and the basis $\{E_i=e^{-\lambda_it/2}e_i\}$ is orthonormal for
the inner product $g_t$. From \eqref{E:ricX}, for the Ricci operator $\ric_t$ of the inner product $g_t$, we get
\begin{equation*}
    \ric_t X =\tfrac12\sum\nolimits_{ijk, \, i>j}e^{(\lambda_k-\lambda_i-\lambda_j)t}\<e_k,[e_i,e_j]\>
    (\<e_k,X\>[e_i,e_j]-\<e_k,[e_i,X]\>e_j+\<e_k,[e_j,X]\>e_i),
\end{equation*}
for any $X \in \g$, so
\begin{equation}\label{E:ricexp}
    \ric_t =\tfrac12\sum\nolimits_{ijk, \, i>j}e^{(\lambda_k-\lambda_i-\lambda_j)t}\<e_k,[e_i,e_j]\>
    ([e_i,e_j]\otimes e_k^* -e_j\otimes e_k^* \ad_{e_i}+e_i\otimes e_k^* \ad_{e_j}),
\end{equation}
where for $X, Y \in \g$, the operator $X \otimes Y^* \in \End(\g)$ is defined by $(X \otimes Y^*)Z=\<Y,Z\>X$.

Let $\Omega=\{(i,j,k) \, : \, 1 \le i,j,k \le n, \; i>j\}$ and let $d=\max_{(i,j,k) \in \Omega}(\lambda_k-\lambda_i-\lambda_j)$ and $\Lambda=\{(i,j,k) \in \Omega \, : \, \lambda_k-\lambda_i-\lambda_j=d\}$. The eigenvectors corresponding to the maximal (minimal) eigenvalues of the operator $\Phi_t=2e^{-td}\ric_t$ are the same as that of $\ric_t$. Taking the limit when $t \to \infty$ we get
\begin{equation}\label{E:Phi0}
    \Phi^0=\lim_{t\to\infty}\Phi_t=\sum\nolimits_{(i,j,k) \in \Lambda} \<e_k,[e_i,e_j]\>
    ([e_i,e_j]\otimes e_k^* -e_j\otimes e_k^* \ad_{e_i}+e_i\otimes e_k^* \ad_{e_j}).
\end{equation}
Suppose the eigenspace of the operator $\Phi^0$ corresponding to the maximal eigenvalue is one-dimensional and is spanned by some $T \in \g$. Then $\pi (T)\in\overline{M}$. Indeed, although the operators $\ric_t$ are not symmetric relative to $\ip$, each of them is symmetric relative to $g_t$. It follows that each of $\ric_t$ is semisimple, with real eigenvalues. The same is true for the operators $\Phi_t$. Moreover, the operators $\Phi_t$ are uniformly bounded for large $t$, hence their eigenvalues also are. As the characteristic polynomial depends continuously on the matrix entries, the maximal eigenvalue of the operator $\Phi^0=\lim_{t \to \infty} \Phi_t$ is the upper limit of the maximal eigenvalues of the $\Phi_t$'s. Now take a sequence of numbers $t_s$ going to infinity and denote $\mu_s$ the maximal eigenvalue of $\Phi_{t_s}$, with $T_s$ a corresponding unit eigenvector (relative to $\ip$). Extracting a subsequence, if necessary, we obtain that the maximal eigenvalue of $\Phi^0$ is $\lim_{s \to \infty}\mu_s$, with the unit vector $T_0=\lim_{s \to \infty}T_s$ a corresponding eigenvector. As the eigenspace of $\Phi^0$ corresponding to the maximal eigenvalue is one-dimensional, we get $\pi(T) = \pi(T_0)\in \overline{M}$ (even though $\Phi^0$ may fail to be semisimple). By the same argument, if the eigenspace of $\Phi^0$ corresponding to the minimal eigenvalue is one-dimensional and is spanned by $T' \in \g$, then $\pi(T')\in \overline{m}$.

Choose the $\lambda_i$'s as follows: $\lambda_1=\dots=\lambda_p=1, \; \lambda_{n-q+1}=\dots=\lambda_n=-1,\; \lambda_i=0$, for $i=p+1, \dots, n-q$, where $p \ge 1, \; q \ge 2$ and $p+q\le n$. Suppose that the skew-symmetric $q \times q$-matrices $J_k$ defined by $(J_k)_{q-n+i,q-n+j}=\<e_k,[e_i,e_j]\>,\; k \le p,\; i, j > n-q$, are linearly independent. We have $\Lambda=\{(i,j,k) \, : \, k \le p, \; i> j > n-q\}$, so from \eqref{E:Phi0},
\begin{equation}\label{E:Phi0J}
    \Phi^0=\tfrac12 \sum\nolimits_{k=1}^p \sum\nolimits_{i,j=n-q+1}^n  (J_k)_{q-n+i,q-n+j}
    ([e_i,e_j]\otimes e_k^* -e_j\otimes e_k^* \ad_{e_i}+e_i\otimes e_k^* \ad_{e_j}),
\end{equation}
hence the matrix of the operator $\Phi^0$ relative to the basis $\{e_i\}$ is
\begin{equation*}
\left(
                   \begin{array}{lll}
                     A & 0_{p \times (n-p-q)} & 0_{p \times q} \\
                     B_1 & 0_{(n-p-q) \times (n-p-q)} & 0_{(n-p-q) \times q}\\
                     B_2 & B_3 & \sum\nolimits_{k=1}^p J_k^2 \\
                   \end{array}
                 \right), \; A_{kl}=\tfrac12 \tr(J_kJ_l^t), \; 1 \le k,l \le p,
\end{equation*}
where $0_{a \times b}$ is the zero ($a \times b$)-matrix, and $B_1,B_2,B_3$ are some matrices of the corresponding dimensions. It follows that the eigenvalues of $\Phi^0$ are the eigenvalues of $A$ (which are positive, as the $J_k$'s are linearly independent, so $A$ is a symmetric positively definite matrix), $0$ (provided $n > p+q$), and the eigenvalues of the symmetric matrix $\sum\nolimits_{k=1}^p J_k^2$ (which are nonpositive, with at least one negative, as that matrix is symmetric nonpositively definite and nonzero). So the maximal eigenvalue of $\Phi^0$ is the maximal eigenvalue of $A$, and the minimal eigenvalue of $\Phi^0$ is the minimal eigenvalue of $\sum\nolimits_{k=1}^p J_k^2$.

\medskip

The proof of assertion~\eqref{IT:eumin}, the only one which deals with the set $m$ of Ricci-minimal vectors, now follows easily: take $p=2, \; q=3$, and suppose that the basis $\{e_i\}$ is chosen in such a way that $e_{n-3+s}=u_s, \; s=1,2,3$. Then by the assumption of assertion~\eqref{IT:eu},
\begin{equation}\label{E:J1J2}
    J_1=\begin{pmatrix}
      0 & a & 0 \\
      -a & 0 & 0 \\
      0 & 0 & 0 \\
    \end{pmatrix}, \;
    J_2=\begin{pmatrix}
      0 & 0 & b \\
      0 & 0 & 0 \\
      -b & 0 & 0 \\
    \end{pmatrix}, \;
    J_1^2+J_2^2=\begin{pmatrix}
      -a^2-b^2 & 0 & 0 \\
      0 & -a^2 & 0 \\
      0 & 0 & -b^2 \\
    \end{pmatrix}.
\end{equation}
As $a, b \ne 0$, the matrices $J_1, J_2$ are linearly independent. So the minimal eigenvalue of $\Phi^0$ is $-a^2-b^2$;
the corresponding eigenspace is one-dimensional and is spanned by $e_{n-2}=u_1$.

\medskip

To treat the remaining assertions \eqref{IT:e1u2}, \eqref{IT:eu}\eqref{IT:e2u3} and \eqref{IT:2s} which deal with the set $M$ of Ricci-maximal vectors, we first compute an eigenvector $T$ of $\Phi^0$ corresponding to the maximal eigenvalue, under the following two assumptions:
\begin{enumerate}[(I)]
  \item \label{it:I}
  that eigenvalue $\lambda_{\max}$ of $A$ (hence of $\Phi^0$) is simple, and %(of multiplicity one)
  \item \label{it:II}
  $\{e_1, \dots, e_p\}$ is the basis of eigenvectors of $A$, with $e_1$ corresponding to $\lambda_{\max}$ (so that $\lambda_{\max}=\tfrac12 \tr(J_1J_1^t)$),
\end{enumerate}
and then give, for each of the three cases \eqref{IT:e1u2}, \eqref{IT:eu}\eqref{IT:e2u3} and \eqref{IT:2s}, the concrete choices of the bases satisfying these assumptions.

We claim that
\begin{equation}\label{E:Tmax}
\begin{gathered}
T=Y + \sum\nolimits_{r=1}^q \eta_r e_{n-q+r}, \quad \text{where} \\
Y=\sum\nolimits_{r,s=1}^q (J_1)_{rs}[e_{n-q+r}, e_{n-q+s}], \qquad \eta=(\lambda_{\max} I_q- \sum\nolimits_{k=1}^p J_k^2)^{-1}\xi,\\
\lambda_{\max}=\tfrac12 \tr(J_1J_1^t), \qquad \xi_r=\sum\nolimits_{s=1}^q \sum\nolimits_{k=1}^p (J_k)_{rs}\<e_k,[e_{n-q+s},Y]\>,
\end{gathered}
\end{equation}
where $\xi, \eta \in \R^q$ and $I_q$ is the $q \times q$ identity matrix. To see that we note that from \eqref{E:Phi0J}, $\Phi^0 Y=\lambda_{\max} Y + \sum\nolimits_{r=1}^q \xi_r \<e_k,[e_{n-q+s},Y]\>e_{n-q+r}$ and $\Phi^0 e_{n-q+r} = \sum\nolimits_{s=1}^q (\sum\nolimits_{k=1}^p J_k^2)_{sr}e_{n-q+s}$, so $\Phi^0 T=\lambda_{\max} Y + \sum_{r=1}^q (\xi+ (\sum_{k=1}^p J_k^2) \eta) = \lambda_{\max}T$, as required (note that $\lambda_{\max} > 0$ and the matrix $\sum\nolimits_{k=1}^p (J_k^2)$ is symmetric nonpositively definite, so the matrix $\lambda_{\max} I_q- \sum\nolimits_{k=1}^p J_k^2$ in the definition of $\eta$ in \eqref{E:Tmax} is indeed nonsingular).

\medskip

We now consider each of the three assertions separately.

\smallskip

To prove \eqref{IT:e1u2}, we take $p=1, \; q=2$, and choose the basis $\{e_i\}$ in such a way that $e_1=e$, $e_{n-1}=u_1, \; e_n=u_2$. Then $J_1=a \left( \begin{smallmatrix} 0 & 1 \\ -1 & 0 \end{smallmatrix} \right)$, where $a=\<e,u_{12}\>$. Suppose that $a \ne 0$. Then assumptions \eqref{it:I} and \eqref{it:II} are trivially satisfied and from \eqref{E:Tmax} we get $Y=2au_{12}, \; \lambda_{\max}=a^2$, $\lambda_{\max} I_2- J_1^2=2a^2I_2, \; \xi_r = \sum\nolimits_{s=1}^2 (J_1)_{rs}\<e,[u_s,Y]\>$, so $\xi=2a^2(\<e,u_{212}\>u_1-\<e,u_{112}\>u_2)$. Then $\eta=(2a^2)^{-1}\xi$, so $T = Y + \eta_1u_1 + \eta_2u_2 = 2\<e,u_{12}\>u_{12} + \<e,u_{212}\>u_1-\<e,u_{112}\>u_2$. We have $\pi(T) \in \overline{M}$, provided $T \ne 0$ and $a=\<e,u_{12}\> \ne 0$. To prove the assertion, it remains to drop the latter condition. But if $\<e,u_{12}\> = 0$, for any three orthonormal vectors $e,u_1,u_2 \in \g$, then $u_{12} \in \Span(u_1,u_2)$, so $\g$ is abelian by Lemma~\ref{L:general}\eqref{IT:br}. This implies $T=0$. Otherwise, the set of triples of orthonormal vectors $e,u_1,u_2 \in \g$, with $\<e,u_{12}\> \ne 0$, is open and dense in the Stiefel manifold $V(3,\g)$, so by continuity and the fact that $\overline{M}$ is closed, $\pi(T)\in \overline{M}$ whenever $T \ne 0$.

\smallskip

For \eqref{IT:2s}, take $p=1,\; q=n-\dim\g'$, and choose the basis $\{e_i\}$ in such a way that $e_1=e \in \g'$ and $e_{n-q+s}=u_s,\; s=1,\dots,q$, is an orthonormal basis for $(\g')^\perp$. Then $(J_1)_{ij}=\<e,u_{ij}\> \ne 0$ by the hypothesis, so assumptions \eqref{it:I} and \eqref{it:II} above are trivially satisfied and from \eqref{E:Tmax} we get $Y=\sum_{ij}\<e,u_{ij}\>u_{ij}$. As $\g$ is two-step nilpotent, $\xi=0$, and hence $\eta=0$. Then $T=Y$ and $\pi(T)\in \overline{M}$, as required.

\smallskip

For \eqref{IT:eu}\eqref{IT:e2u3}, take $p=2, \; q=3$, and choose the basis $\{e_i\}$ in such a way that $e_{n-3+s}=u_s$ for $s=1,2,3$. Then by the hypothesis, the matrices $J_1,J_2$ are given by \eqref{E:J1J2}, with $|a|>|b|>0$ and $A=\mathrm{diag}(a^2, b^2)$, so assumptions \eqref{it:I} and \eqref{it:II} are satisfied with $\lambda_{\max}=\frac12 \tr(J_1J_1^t)=a^2$. Then from \eqref{E:Tmax} we get $\lambda_{\max} I_3-\sum\nolimits_{k=1}^p J_k^2=\mathrm{diag}(2a^2+b^2, 2a^2,a^2+b^2)$. Furthermore, $Y=2au_{12}$, so $\xi_1=2a^2\<e_1,u_{212}\>+2ab\<e_2,u_{312}\>, \; \xi_2=-2a^2\<e_1,u_{112}\>$ and $\xi_3=-2ab\<e_2,u_{112}\>$.
Then $T=Y+\sum_{i=1}^3 \eta_iu_i= 2a T_2$. As $a \ne 0$, it follows that $\pi(T_2) \in \overline{M}$.

To show that $\pi(T_1) \in \overline{M}$ when $T_1 \ne 0$, consider a one-parameter family $h_t$ of inner products on $\g$ defined by the requirement that the basis $\{E_1=te_1, E_2=e_2, \dots, E_n=e_n\}$ is orthonormal. Then $h_t(E_1,X)=t^{-1}\<e_1,X\>, \; h_t(Y,X)=\<Y,X\>$ for all $X \in \g$ and all $Y \perp e_1$, and the hypothesis of the assertion is satisfied, provided $0 < t < |b^{-1}a|$. For each such $t$, the projection to $\mathbb{P}\g$ of the corresponding vector $T_2=T_2(t)$ belongs to $\overline{M}$. Computing $T_2(t)$ and taking the limit when $t \to |b^{-1}a|^{-}$ we obtain $\frac{1}{6ab}T_1$, so $\pi(T_1) \in \overline{M}$.

\section{Proof of Lemmas~\ref{L:general}, \ref{L:notL} and \ref{L:structure}}
\label{S:lemmata}

\begin{proof}[Proof of Lemma~\ref{L:general}]
\eqref{IT:br} If $[X, Y]=aX+bY$, then $[X,[X, Y]]=b[X,Y]$. As $\ad_X$ is nilpotent, $b=0$. Similarly, $a=0$.

\eqref{IT:2br} If $[X,[X, Y]] \ne 0$, then for some $l \ge 3$, we have $(\ad_X)^lY=0$, but $(\ad_X)^{l-1}Y \ne 0$. Acting on the both sides of the equation $[X,[X, Y]]=aX+bY+c[X,Y]$ by $(\ad_X)^{l-1}$ we get $b=0$; acting by $(\ad_X)^{l-2}$ we get $c=0$. Then $a$ is an eigenvalue of $\ad_{[X,Y]}$, so $a=0$.

\eqref{IT:two}  Polarizing the equation $[X,[X, Y]]=0$ we get $[X,[Z, Y]]+[Z,[X, Y]]=0$, for all $X, Y, Z \in \g$. Then from the Jacobi identity, $[Y,[X,Z]]=-[X,[Z, Y]]-[Z,[Y, X]]=2[X,[Y, Z]]$. Interchanging $X$ and $Y$ gives $[X,[Y, Z]]=0$, so $\g$ is two-step nilpotent.
\end{proof}

\begin{proof}[Proof of Lemma~\ref{L:notL}]
If $\dim \g'=1$, then $\g$ is the direct product of a Heisenberg algebra and an abelian ideal (which can be trivial). We can therefore assume that $\dim \g' > 1$.

First suppose that $\operatorname{codim} \g' > 2$. Let $[X,Y]=U \ne 0$ for some $X, Y \in \g$, which, by a small perturbation, can be chosen in such a way that $\Span (X,Y) \cap \g'=0$. Let $Z \notin \Span(X,Y)$ be chosen in such a way that $\Span (X,Y,Z) \cap \g'=0$ (as $\operatorname{codim} \g' > 2$, the set of such $Z$ is open and dense in $\g$). Then the vectors $X,Y,Z,[X,Y]=U$ are linearly independent, so by hypothesis, the vector $[X,Z]$ belongs to their linear span: $[X,Z]=a_1X+a_2Y+a_3Z+bU$. As $U \in \g'$ and $\Span (X,Y,Z) \cap \g'=0$ it follows that $a_1=a_2=a_3=0$, so $[X,Z] \parallel U$. Thus for all $Z$ from an open and dense subset $\mathcal{U} \subset \g$, the vector $[X,Z]$ is nonzero and is parallel to $U$ and $\Span (X,Z) \cap \g'=0$. Then for every $Z \in \mathcal{U}$,  the above arguments applied to the pair $(Z,X)$ in place of $(X,Y)$ tell us that for all $W$ from an open, dense set $\mathcal{U}_Z \subset \g$, the vector $[Z,W]$ is parallel to $[X,Z]$, hence is parallel to the fixed vector $U$. It follows that $\g'= \R U$, a contradiction.

Now consider the case $\operatorname{codim} \g' \le 2$. Then $\g$ is generated by two elements $e_1, e_2$ which can be chosen arbitrarily to satisfy $\Span(e_1,e_2) \oplus \g'=\g$. Then $e_3=[e_1,e_2]$ is nonzero (otherwise $\g$ is abelian) and the vectors $e_1,e_2,e_3$ are linearly independent by Lemma~\ref{L:general}\eqref{IT:br}, so by hypothesis, some nontrivial linear combination $a[e_1,e_3]+b[e_2,e_3]$ lies in $\Span(e_1,e_2,e_3)$. Without loss of generality we may assume that $a \ne 0$. Denote $f_1=e_1+a^{-1}be_2, \; f_2=e_2$. Then the elements $f_1, f_2$ still generate $\g$ and we have $f_3:=[f_1,f_2]=e_3$, so $[f_1,f_3] =[f_1,[f_1,f_2]] \in \Span(f_1,f_2,f_3)$, so $[f_1,[f_1,f_2]]=0$ by Lemma~\ref{L:general}\eqref{IT:2br}. As $\dim \g'\ge 2$, we must have $f_4:=[f_2,f_3]=[f_2,[f_1,f_2]] \ne 0$, so by Lemma~\ref{L:general}\eqref{IT:2br}, the vectors $f_1,f_2,f_3,f_4$ are linearly independent. Then by the Jacobi identity, $[f_1,f_4]=[f_1,[f_2,[f_1,f_2]]]=[f_2,[f_1,[f_1,f_2]]]=0$, as $[f_1,[f_1,f_2]]=0$.  By hypothesis, we have $\rk(f_1,f_2,f_3,f_4,[f_2,f_4])=\dim\mathcal{L}(f_1,f_2,f_4) \le 4$, so $[f_2,f_4]$ must be a linear combination of $f_1,f_2,f_3,f_4$. Then, since $[f_2,f_4] \in \g'$ and $\Span(f_1,f_2)\oplus \g'=\g$, the vector $[f_2,f_4]=[f_2,[f_2,f_3]]$ is a linear combination of $f_3$ and $ f_4=[f_2,f_3]$. So $[f_2,f_4]=0$, by Lemma~\ref{L:general}\eqref{IT:2br}. Therefore, $[f_1,f_4]=[f_2,f_4]=0$. Then $[f_3,f_4]=[[f_1,f_2],f_4]=0$, by the Jacobi identity. As $f_1, f_2$ generate $\g$, it follows that $\g=\Span(f_1,f_2,f_3,f_4)$, with the brackets $[f_1,f_2]=f_3, \; [f_2, f_3]=f_4, \; [f_1,f_3]=[f_1,f_4]=[f_2,f_4]=[f_3,f_4]=0$. Hence $\g$ is the four-dimensional filiform algebra, as required.
\end{proof}
We note in passing that the hypothesis of Lemma~\ref{L:notL} is trivially satisfied when $\dim \g \le 4$, so we get yet another classification of nilpotent algebras of dimension up to four.

\begin{proof}[Proof of Lemma~\ref{L:structure}]
Since $\g$ is not two-step nilpotent, Lemma~\ref{L:general} implies that $[X,[X, Y]] \notin \Span(X,Y,[X,Y])$ for almost all $(X, Y) \in \g^2$. Take one such pair $(X,Y)$. As conditions \eqref{E:rk5} is violated, we have $[X,[X, Y]] + a [Y,[X, Y]] \in \Span(X,Y,[X,Y])$ for some $a \ne 0$, so $[\tilde X,[\tilde X, Y]] \notin \Span(\tilde X,Y,[\tilde X,Y])$, where $\tilde X = X+aY$. By Lemma~\ref{L:general}\eqref{IT:2br}, $[\tilde X,[\tilde X, Y]]=0$, so $[X,[X, Y]] + a [Y,[X, Y]]=0$. It follows that $[X,[X, Y]] \parallel [Y,[Y, X]]$, for all $(X, Y) \in \g^2$. Choose a basis $\{e_i\}$ for $\g$. Let $x_j$ and $y_j$ be the components of $X$ and $Y$ relative to this basis. Denote $\mathbf{K}=\R[x_1, \dots , x_n, y_1, \dots, y_n]$. The components of the vector $[X,[X, Y]]$ relative to $\{e_i\}$ are polynomials $P_i(X,Y) \in \mathbf{K}$. Note that at least one of the $P_i$'s is nonzero by Lemma~\ref{L:general}\eqref{IT:two} and that every nonzero $P_i$ is homogeneous of degree $2$ in the $x_j$'s and homogeneous of degree $1$ (linear) in the $y_k$'s. Moreover, as $[X,[X, Y]] \parallel [Y,[Y, X]]$, there exist nonzero polynomials $f(X,Y)$ and $h(X,Y)$ such that
  \begin{equation}\label{E:propor}
  f(X,Y)[X,[X, Y]]=h(X,Y)[Y,[Y, X]]
  \end{equation}
(for instance, if $P_i \ne 0$, one can take $f(X,Y)=P_i(Y,X),\; h(X,Y)=P_i(X,Y)$). Cancelling the common factor, if necessary, we can assume that $f$ and $h$ in \eqref{E:propor} are coprime over $\mathbf{K}$. Note that $h$ is a nonconstant polynomial, as the left-hand side of \eqref{E:propor} is of degree at least two in the coordinates of $X$. Then from \eqref{E:propor}, every polynomial $P_i$ is divisible by $h$, so $P_i(X,Y)=h(X,Y)Q_i(X,Y)$ for some $Q_i\in \mathbf{K}$, hence
  \begin{equation}\label{E:Psi}
    [X,[X, Y]]=h(X,Y) Q(X,Y) \ne 0, \quad \text{where} \; Q=(Q_1, \dots, Q_n) \in \mathbf{K}^n.
  \end{equation}
Moreover, by moving the greatest common divisor of the $Q_i$'s to $h$, we can assume that $\gcd(Q_1, \dots, Q_n)=1$. From homogeneity of the $P_i$'s, it follows that $h$ is homogeneous of degree $d_1=0,1,2$ in the $x_j$'s and homogeneous of degree $d_2=0,1$ in the $y_k$'s. Then every nonzero $Q_i$ is homogeneous of degree $2-d_1$ in the $x_j$'s and of degree $1-d_2$ in the $y_k$'s. Moreover, $(d_1,d_2) \ne (0,0)$, as $h$ is nonconstant. Furthermore, $(d_1,d_2) \ne (0,1)$, as otherwise from \eqref{E:Psi} we would get $[X,[X, Y]]=h(Y) Q(X)$, which would imply $h(X) Q(X)=0$, so either $h$ or $Q$ vanish, and then so does $[X,[X, Y]]$. By a similar argument, $(d_1,d_2) \ne (2,0)$. In the case $(d_1,d_2) = (1,1)$, equation~\eqref{E:Psi} implies that $[X,[X, Y]]=h(X,Y) Q(X)$, where $h$ is a nonzero bilinear form and $Q \in \End(\g)$. Then from $[X,[X, Y]] \parallel [Y,[Y, X]]$, we obtain $Q(X) \parallel Q(Y)$, for all $X, Y \in \g$, which means that $Q$ is the tensor product of a nonzero linear form and a constant vector. This contradicts the fact that $\gcd(Q_1, \dots, Q_n)=1$.

Therefore only the following two cases may occur: $(d_1,d_2)=(1,0),\,(2,1)$.

\medskip

In the first case, $[X,[X, Y]]=h(X) Q(X,Y)$, for some nonzero linear form $h$ and a bilinear map $Q:\g^2 \to \g$ (which must be skew-symmetric). Let $\h = \ker h$ and let $e$ be the vector dual to $h$ with respect to some inner product $\ip$ on $\g$ (the latter being scaled in such a way that $e$ is unit). Then $h(X)=\<X,e\>$ and $[X,[Z, Y]]+[Z,[X, Y]]=\<X,e\> Q(Z,Y)+\<Z,e\>Q(X,Y)$, by polarisation. Subtracting from this the same equation, with $(X,Y,Z)$ cyclicly permuted and using the Jacobi identity we get
  \begin{equation}\label{E:bracket}
    [X,[Y,Z]]=-\tfrac13 \<Z,e\>Q(X,Y)+\tfrac13 \<Y,e\>Q(X,Z)+\tfrac23 \<X,e\>Q(Y,Z).
  \end{equation}
It follows that $[\h,[\h,\h]]=0$. We claim that $\h$ is an ideal. As $\operatorname{codim} \h =1$ it suffices to show that $\h$ is a subalgebra. Suppose there exist $Y,Z \in \h$ such that $[Y,Z] \notin \h$, that is, $\<[Y,Z],e\> \ne 0$. Then for arbitrary $W,X \in \h$ equation \eqref{E:bracket} and the fact that $[\h,[\h,\h]]=0$ give $0=[W,[X,[Y,Z]]] =-\tfrac13\<[Y,Z],e\>Q(W,X) =\<[Y,Z],e\>[W,[X,e]]$, so $[\h,[\h,e]]=0$ and hence by the Jacobi identity, $[e,[\h,\h]]=0$ which implies that $[\h,\h] \subset \z$. Then for any $X \in \h$, we have $0=[e,[X,[Y,Z]]]=[e,[X,\<[Y,Z],e\>e]]$, so $[e,[\h,e]]=0$, which means that $\g$ is two-step nilpotent, a contradiction. Therefore $\h$ is a two-step nilpotent ideal of $\g$ and the operator $D=\ad_e$ restricted to $\h$ acts as a nilpotent derivation. Now from \eqref{E:bracket}, for any $X \in \h$ we get $[X,[X,e]]=-\tfrac13 Q(X,X)=0$, as $Q$ is skew-symmetric, so $[DX,X]=0$, so $\g$ is an algebra of class \eqref{IT:c}.
%So $\h$ is a subalgebra of $\g$, which implies that $\h$ is an  ideal (if $[X,e]=\lambda e \mod \h$ for some $X \in \h$, then acting by $\ad_X^N$, with $N$ large enough, we get $0=\ad_X^{N+1}e=\lambda^Ne \mod \h$, so $\lambda=0$).

\medskip

In the second case, when $(d_1,d_2)=(2,1)$ in \eqref{E:Psi}, we get $[X,[X, Y]]=h(X,Y) c$, for some nonzero $c \in \g$ and a nonzero polynomial function $h$ which is homogeneous of degree $d_1=2$ in the $x_j$'s and is linear in the $y_k$'s. We can assume that the polynomial $h(X,Y)$ is not divisible by any linear form $\tilde h(X)$, as otherwise $[X,[X, Y]]=\tilde h(X) (h(X,Y)(\tilde h(X))^{-1} c)$ and we get back to the case $(d_1,d_2)=(1,0)$ considered in the previous paragraph.

We have $[X,[X,[X, Y]]]=h(X,[X,Y]) c=h(X,Y)[X,c]$, so by Lemma~\ref{L:general}\eqref{IT:br}, $[X,c]=0$ for all $X$ such that $h(X,Y) \ne 0$ for at least one $Y \in \g$. As $h \ne 0$, this holds for almost all $X \in \g$, so $c \in \z$ (and hence $h(X,c)=h(c,X)=0$, for all $X \in \g$). Factoring out the ideal $\R c$ we get an $(n-1)$-dimensional nilpotent algebra $\h$ whose Lie bracket we denote $[\cdot,\cdot]_{\h}$. We have $[X,[X, Y]_{\h}]_{\h}=0$, for all $X, Y \in \h$, so $\h$ is two-step nilpotent by Lemma~\ref{L:general}\eqref{IT:two} and $\g$ is a central extension of $\h$ by a cocycle $\omega$. Introduce an arbitrary inner product $\ip$ on $\g$ and identify $\h$ with the subspace $(\R c)^\perp$. Then $[X,Y]=[X,Y]_{\h}+\omega(X,Y)c$ for $X,Y \in \h$.

As for $X, Y \in \h$ we have $\omega(X,[X,Y]_\h)=h(X,Y)$, it remains to show that for almost all $X \in \h$, there exists $Y \in \h$ such that $h(X,Y)=0$ and $h(Y,X) \ne 0$. Suppose $\mathcal{U} \subset \h$ is an open subset such that for all $X \in \mathcal{U}$ this property is violated. Define a map $v:\h \to \h$ by $\<v(X),Y\>=h(X,Y)$, for $Y \in \h$. Note that $v(X)$ is not identically zero, so $v(X) \ne 0$ for almost all $X \in \h$. Replacing $\mathcal{U}$ by a smaller open subset, if necessary, we can assume that $v(X) \ne 0$, for all $X \in \mathcal{U}$. For every $X \in \h$, let $S_X \in \End(\h)$ be a symmetric operator defined by $\<S_XY,Y\>=h(Y,X)$ (note that $S_X$ is not identically zero). Then for all $X \in \mathcal{U}$, we have $Y \perp v(X) \implies \<S_XY,Y\>=0$, which implies that there exists a map $\tilde T: \mathcal{U} \to \h$ such that for all $X \in \mathcal{U}, \; S_X=\tilde T(X) \otimes v(X)^*+ v(X) \otimes \tilde T(X)^*$. The components $v_i(X)$ of the vector $v(X)$ relative to an orthonormal basis $\{e_i\}$ for $\h$ are quadratic forms in $x_i$, the coordinates of $X$ relative to $\{e_i\}$. By a small perturbation of the basis, we can assume that all of them are not (identically) zero. Then on $\mathcal{U}, \; (S_X)_{ii}=2 (\tilde T(X))_i v_i(X)$, so the $(\tilde T(X))_i$'s are rational functions on $\mathcal{U}$. Clearing the denominators we get
  \begin{equation}\label{E:KXf}
  f(X)S_X= T(X) \otimes v(X)^*+ v(X) \otimes T(X)^*
  \end{equation}
for some nonzero polynomial function $f(X)$ and some polynomial vector $T(X)$. Polynomial equation \eqref{E:KXf} holds for all $X$ from an open set $\mathcal{U} \subset \h$, hence for all $X \in \h$. Dividing both sides by $d(X)=\gcd(f(X), T_1(X), \dots, T_{n-1}(X))$ and replacing $f(X)$ by $f(X)/d(X)$ and $T_i(X)$ by $T_i(X)/d(X)$, we can assume that $\gcd(f(X), T_1(X), \dots, T_{n-1}(X))=1$. Then for any prime factor $\tilde f$ of $f$ over $\R[x_1, \dots, x_{n-1}]$, there exists $i=1, \dots, n-1$ such that $T_i$ is not divisible by $\tilde f$. But then from \eqref{E:KXf}, $\tilde f \mid T_i v_i$, so $\tilde f \mid v_i$, and $\tilde f \mid T_i v_j+T_jv_i$, for all $j \ne i$, so $\tilde f \mid v_j$, hence all the components of $v$ are divisible by $\tilde f$. Dividing both sides of \eqref{E:KXf} by $\tilde f$ and repeating the arguments we obtain that all the $v_i$'s are divisible by $f$. As all the $v_i$'s are quadratic forms, $f$ must be a homogeneous polynomial of degree $d \le 2$. But $f$ cannot be a constant, as then the left-hand side of \eqref{E:KXf} is a matrix whose entries are linear forms of $X$, while the degree of the right-hand side is at least two ($T(X)$ is not identically zero, as $S_X$ is not). Moreover, $f$ cannot be linear, as otherwise polynomial $h(X,Y)=\<v(X),Y\>$ is divisible by a linear form $f(X)$ contradicting the assumption made above. So $f$ must be a quadratic form. But this again leads to a contradiction, as then $v(X)=f(X) c_0$, where $c_0 \ne 0$ is some constant vector from $\h$, so $0=h(X,X)=\<v(X),X\>=f(X)\<c_0, X\>$, for all $X \in \h$, so either $f$ or $c_0$ is zero, hence $v(X)=0$.

It follows that the set of $X \in \h$ for which there exists $Y \in \h$ such that $\omega(X,[X,Y]_\h)=0$ and $\omega(Y,[X,Y]_\h) \ne 0$ is dense, so $\g$ is an algebra of class~\eqref{IT:b}.

\medskip

We now assume that $\g$ belongs to class~\eqref{IT:c} and prove the remaining part of lemma. Clearly, for almost all $(X_1,X_2,X_3) \in \g^3$ the subspace $L:=\mathcal{L}(X_1,X_2,X_3)$ has the same dimension $N$. If $N \le 4$, then by Lemma~\ref{L:notL} the algebra $\g$ is either abelian or two-step nilpotent or has a codimension one abelian ideal, which contradicts the hypothesis. It follows that $N = 5, 6$. In the both cases we will need the following two observations. We have $\g=\R c \oplus \h$, where $\h$ is a two-step nilpotent ideal and $D=\ad_{c|\h}$ is a nilpotent derivation of $\h$ satisfying $[DU,U]=0$, for all $U \in \h$. Then $[DU,V]=[U,DV]$, so since $D$ is a derivation, we obtain
\begin{equation}\label{E:duv}
    D[U,V]=2[DU,V]=2[U,DV], \quad \text{for all}\; U, V \in \h.
\end{equation}
Furthermore, if we choose a triple $(X_1, X_2, X_3) \in \g^3$, with $\dim L=N$ and denote $W=\Span(X_1,X_2,X_3)$, then $\dim W = 3$. As both the hypothesis and the claim depend only on $W=\Span(X_1,X_2,X_3)$, rather than on the triple $(X_1,X_2,X_3)$ itself, we can assume that for almost all $W$ in the Grassmannian $G(3,\g)$ we have $\dim (W + [W,W])=N$. We can additionally assume that $[W,[W,W]] \ne 0$ (otherwise every three-fold bracket in $\g$ would be zero, so $\g$ would be two-step nilpotent). Moreover, as the codimension one ideal $\h \subset \g$ is two-step nilpotent, we get $W \not\subset \h$, so we can choose a basis $\{c, X, Y\}$ for $W$ in such a way that $X,Y \in \h,\; c\notin \h$ (note that it does not matter, which $c \notin \h$ to choose to define the derivation $D = (\ad_c)_{|\h}$). Furthermore we can assume that $[X, Y] \ne 0$, as $\h$ is not abelian.

We now consider the two cases $N=5,6$ separately.

\medskip

Suppose $N=6$, so that for almost all triples $(X_1, X_2, X_3) \in \g^3$, we have $\dim L=6$. As condition \eqref{E:rk7} is violated, $X_{312}$ lies in $L$. Moreover, as this still holds if we replace $X_1,X_2,X_3$ by any three vectors spanning the same three-space $W$, any triple bracket of the vectors $X_1,X_2,X_3$ lies in $L$. It follows that $L$ is a six-dimensional subalgebra of $\g$ generated by any basis for $W$, so $L=W+[W,W]$ which in particular implies that $\dim L' \ge 3$. Furthermore, as the above argument shows, $L$ itself is a Lie algebra of class~\eqref{IT:c}, so $L=\R c \oplus \m$, where $\m=\h \cap L$ is a five-dimensional two-step nilpotent (nonabelian) ideal and $D=\ad_{c|\m}$ is a nilpotent derivation of $\m$ satisfying $[DU,U]=0$ and \eqref{E:duv}, for all $U,V \in \m$.

One can now browse through one of the classification lists of six-dimensional nilpotent Lie algebras available in the literature (the three algebras of assertion~\eqref{IT:six} are $L_{6,19}(0)$, $L_{6,23}$ and $L_{6,12}$ respectively from \cite{dG}, or $\g_{6,19}, \g_{6,20}$ and $\g_{6,2}$ respectively from \cite{Niel}). We will use the classification of five-dimensional algebras instead, for the ideal $\m$.

Up to an isomorphism, there are three two-step nilpotent nonabelian five-dimensional Lie algebras: the Heisenberg algebra $\h_5$, the direct product of the Heisenberg algebra $\h_3$ and the abelian ideal $\a_2$, and the algebra $\Span(X,Y_1,Y_2,Z_1,Z_2)$ defined by the relations $[X,Y_1]=Z_1, \; [X,Y_2]=Z_2$ (the algebras $L_{5,4}, L_{5,2}$ and $L_{5,8}$ from \cite{dG} respectively). % or Nielsen, but I don't have 5-dim

We first observe that $\m$ cannot be isomorphic to the latter algebra. Indeed, arguing by contradiction, from $[DX,X]=0$ we get $DX \in \Span(X,Z_1,Z_2)=\R X \oplus \m'$. As $D$ is nilpotent and $D\m' \subset \m'$, we get $DX \in \m'$. Then from \eqref{E:duv} with $U=X$ we obtain $[X,D\m]=0$, so $D\m \subset \m'$. It follows that $L' = \m'$ contradicting the fact that $\dim L' \ge 3$.

Similarly, $\m$ cannot be isomorphic to the Heisenberg algebra $\h_5$ (defined by the relations $[X_1,Y_1]=Z, \; [X_2,Y_2]=Z$). Otherwise, as $D$ is nilpotent and $\m'=\R Z$, we get $DZ=0$. Then \eqref{E:duv} implies $[D\m,\m]=0$, so $D\m \subset \R Z$, hence $\dim L' = 1$, a contradiction.

The only possible case is therefore when $\m$ is the direct product of the Heisenberg algebra $\h_3$ and the abelian ideal $\a_2$, so $\m=\Span(X,Y,Z,A_1,A_2)$, with the only nonzero bracket $[X,Y]=Z$. Then $\z(\m)=\Span(A_1,A_2,Z)$ and $\m'=\R Z$. From the fact that both these subspaces are $D$-invariant and that $D$ is nilpotent, we get $DZ=0$ and $DA_1=\rho A_2 + \sigma Z,\; DA_2=\theta Z$, for some $\rho, \sigma, \theta \in \R$ (changing the basis for $\Span(A_1, A_2)$ if necessary). Then the left-hand side of \eqref{E:duv} vanishes identically, so $D\m \in \z(\m)$, that is, $DX= \alpha_1 A_1 + \alpha_2 A_2 + \gamma Z, \; DY=\beta_1 A_1 + \beta_2 A_2 + \delta Z$, for some $\alpha_i, \beta_i, \gamma, \delta \in R$. Then $L'=\R Z + D\m=\Span(Z, \rho A_2, \alpha_1 A_1 + \alpha_2 A_2, \beta_1 A_1 + \beta_2 A_2)$, so from $\dim L' \ge 3$ we obtain $\rk\left( \begin{smallmatrix} 0 & \alpha_1 & \beta_1 \\ \rho & \alpha_2 & \beta_2 \end{smallmatrix} \right)=2$.

If $\rho = 0$, the matrix $Q=\left( \begin{smallmatrix} \alpha_1 & \beta_1 \\ \alpha_2 & \beta_2 \end{smallmatrix} \right)$ is nonsingular. Replacing $c$ by $c+\gamma Y - \delta X$ we can assume that $\gamma=\delta=0$. As $\sigma$ and $\theta$ cannot be both zero (as otherwise $L$ is two-step nilpotent) we can assume that $DA_1= \sigma Z,\; DA_2=0, \; \sigma \ne 0$, and then, by a change of basis for $\Span(X, Y)$, that $DX= \alpha_1 A_1, \; DY= \beta_2 A_2, \; \alpha_1 \beta_2 \ne 0$. Now changing the elements of the basis $\{X,Y,A_1,A_2,Z\}$ to $\{X, \alpha_1^{-1} \sigma^{-1}Y, \alpha_1 A_1, \alpha_1^{-1} \sigma^{-1} \beta_2 A_2, \alpha_1^{-1} \sigma^{-1}Z\}$ respectively we get the relations of algebra \eqref{E:dimsix2}.

Suppose $\rho \ne 0$. Replacing $X, Y$ by $X-\rho^{-1}\alpha_2 A_1, Y-\rho^{-1}\beta_2 A_1$ respectively we can assume that $\alpha_2 = \beta_2 = 0$. Then changing the basis for $\Span(X, Y)$ we can further assume that $\beta_1=0$ (and so $\alpha_1 \ne 0$, as otherwise $\dim L' <3)$. Replacing $c$ by $c+\gamma Y - \delta X$, $A_1$ by $\rho^{-1}A_1$ and $A_2$ by $A_2 + \sigma Z$, we get $DX= \alpha A_1, \; DY=0, \; DA_1= A_2,\; DA_2=\theta Z, \; DZ=0$, where $\alpha=\alpha_1 \rho \ne 0$. Now replacing $X, Y$ by $\alpha^{-1}X, \alpha Y$ respectively we obtain $DX= A_1, \; DY=0$, $DA_1= A_2,\; DA_2=\theta Z, \; DZ=0$.
If $\theta = 0$ we get algebra \eqref{E:dimsix1}. If $\theta \ne 0$, we replace $Z, Y$ by $\theta Z, \theta Y$ respectively, which gives algebra \eqref{E:dimsix3}.

\medskip

Suppose $N=5$, so that for almost all triples $(X_1, X_2, X_3) \in \g^3$, we have $\dim L=5$ (note that condition \eqref{E:rk7} is then trivially violated). Choose one such triple and denote $W=\Span(X_1,X_2,X_3)$. Then $\dim W = 3$, and as it was shown above, we can assume that for almost all $W$ in the Grassmannian $G(3,\g)$ we have $\dim (W + [W,W])=5$ and $[W,[W,W]] \ne 0$. Furthermore, we can choose a basis $\{c, X, Y\}$ for $W$ in such a way that $X,Y \in \h,\; c\notin \h$ and $[X, Y] \ne 0$.

We first prove that $L$ is a subalgebra of $\g$. As $\rk(c,X,Y,[X,Y],DX,DY) =5$ and $c \notin \h$, we get $\rk(X,Y,[X,Y],DX,DY) =4$, so $\alpha_1 X + \alpha_2 Y + \gamma [X,Y] + \beta_1 DX + \beta_2 DY=0$ for some $\alpha_i, \gamma, \beta_i \in \R$ not all of which are zero. Taking the bracket with $X$ we get $\alpha_2 [X,Y] + \frac12 \beta_2 D[X,Y]=0$, by \eqref{E:duv}. As $D$ is nilpotent and $[X,Y] \ne 0$, it follows that $\alpha_2=0$ and $\beta_2 D[X,Y]=0$. Similarly $\alpha_1=0$ and $\beta_1 D[X,Y]=0$. If $D[X,Y] \ne 0$, then  $\beta_1=\beta_2=0$, hence $\gamma [X,Y] = 0$, a contradiction. Therefore $D[X,Y]=0$ and $\gamma [X,Y] + \beta_1 DX + \beta_2 DY=0$, with at least one of $\beta_i$ nonzero; suppose $\beta_2 \ne 0$. Then acting by $D$ we get $D^2Y=-\beta_2^{-1}\beta_1 D^2X$. Let $Z=D^2X$. If $Z=0$, then $D^2X=D^2Y=0$ and from the fact that $X, Y \in \h, \; D[X,Y]=0$ and equation \eqref{E:duv}, we find that $[W,[W,W]] = 0$, a contradiction. It follows that $Z \ne 0$, which implies that the vectors $[X,Y]$ and $DX$ are linearly independent (as both are nonzero and $[X,Y] \in \ker D$, while $DX \notin \ker D$). Now, for almost all $t \in \R$, the above arguments work for the space $W_t=\Span(c,X,Y+tDX)$ in place of $W$. Taking such a $t$ (and using the fact that $[X,DX]=0$) we find that $\gamma(t)[X,Y] + \beta_1(t) DX + \beta_2(t)(DY+tZ)=0$, for some $\gamma(t), \beta_i(t) \in \R$, which are not all zeros. Moreover, as the vectors $[X,Y]$ and $DX$ are linearly independent, we get $\beta_2(t) \ne 0$, so $Z=-t^{-1}(\beta_2(t)^{-1}(\gamma(t)[X,Y] + \beta_1(t) DX) + DY) \in [W,W]$. It follows that $D^2X, D^2Y\in [W,W]$. Now the fact that $h$ is two-step nilpotent, equation $D[X,Y]=0$ and \eqref{E:duv} imply that all the other three-fold brackets of elements of $W$ vanish, so we obtain $[W,[W,W]] \subset [W,W]$, hence $L=W+[W,W]$ is a subalgebra.

Furthermore, the above argument shows that the subspace $\Span([X,Y],DX,DY)$ has dimension two and contains the vector $Z=D^2X$. As $Z$ and $A:=DX$ are linearly independent by Lemma~\ref{L:general}\eqref{IT:2br}, we get $\Span([X,Y],DX,DY)=\Span(A, Z)$, so $L=\Span(c,X,Y,[X,Y],DX,DY)=\Span(c,X,Y,A,Z)$. To find relations for $L$ we first consider the two-step nilpotent subalgebra $\m:=\Span(X,Y,A,Z)=L \cap \h$. We know that $[X,Y]=pA+qZ \ne 0$, so from $D[X,Y]=0$ it follows that $pZ+qDZ=0$, hence $p=0$ and $DZ=0$ since $D$ is nilpotent. Then $q \ne 0$ and by scaling $Y$ we get $[X,Y]=Z$. Hence $Z \in \z(\m)$ and also $[X,A]=[A,Z]=0$ and $[Y,A]=[Y,DX]=-\frac12 D[X,Y]=0$ from \eqref{E:duv}. So the only nonzero bracket in $\m$ is $[X,Y]=Z$. Furthermore, we have $[c,X]=A,\; [c,A]=Z$ by construction and $[c,Z]=DZ=0, \; [c,Y]=DY=r A+s Z$ from the above. Replacing $Y$ by $Y-p_1 X - q_1 A$ we get $[c,Y]=0$, without violating the fact that $[X,Y]=Z$.
This gives the required relations, for almost all choices of $W \in G(3,\g)$.
\end{proof}

\begin{remark}
It can be shown that any Lie algebra $\g$ satisfying Lemma~\ref{L:structure}\eqref{IT:five} (that is, an algebra $\g$ of class~\eqref{IT:c} with $N=5$) is isomorphic to one of the following algebras (which are one-dimensional extensions of the direct sum of a Heisenberg algebra and an abelian ideal):
$\g\cong\Span(X_1, \dots, X_k, Y_1, \dots, Y_k, Z_1, \dots, Z_l)$, with $k \ge 2$, defined by the relations
\begin{equation*}
     [X_i,Y_i]=Z_1, \; [X_1, Z_2]=Y_1, \qquad \text{or} \qquad[X_i,Y_i]=Z_1, \; [X_1, X_2]=Y_1,
\end{equation*}
where $l \ge 2$ in the first case and $l \ge 1$ in the second case. We do not use this fact in the proof of Theorem~\ref{T:maxmin}.
\end{remark}

\bibliographystyle{amsplain}
\bibliography{chnn}

\end{document}